\documentclass[11pt,oneside,english,a4paper]{amsart}
\usepackage[T1]{fontenc}
\usepackage[utf8]{inputenc}
\usepackage[dvipsnames]{xcolor}
\usepackage[yyyymmdd,hhmmss]{datetime}
\usepackage[numbers,sort,compress]{natbib}
\usepackage{amstext,amsthm,amssymb}
\usepackage{thmtools}
\usepackage{xparse,mathrsfs,mathtools}
\usepackage[english]{babel}
\usepackage[
	unicode=true,
	pdfusetitle, pdfdisplaydoctitle=true,
	pdfpagemode=UseOutlines,
	bookmarks=true, bookmarksnumbered=false, bookmarksopen=true, bookmarksopenlevel=2,
	breaklinks=true,
	pdfencoding=unicode, psdextra,
	pdfcreator={},
	pdfborder={0 0 0}
]{hyperref}
\usepackage{lmodern,microtype}
\usepackage[a4paper]{geometry}
\usepackage[capitalise,nameinlink]{cleveref}
\usepackage[cal=dutchcal,bb=ams,scr=boondoxo]{mathalfa}
\usepackage[inline]{enumitem}
\usepackage{subcaption}
\usepackage{zref-clever}
\zcsetup{hyperref=true, nameinlink=true, comp, cap, nosort}
\usepackage{tikz}
\usepackage{pgfplots}
\usetikzlibrary{backgrounds,calc,cd,spy}
\pgfdeclarelayer{fg}
\pgfsetlayers{background,main,fg}
\pgfplotsset{compat=1.18}

\newtheorem{thm}{Theorem}[section]
\AddToHook{env/thm/begin}{\zcsetup{countertype={thm=theorem}}}

\newtheorem{lem}[thm]{Lemma}
\AddToHook{env/lem/begin}{\zcsetup{countertype={thm=lem}}}
\zcRefTypeSetup{lem}{
	Name-sg = {Lemma},
	name-sg = {lemma},
	Name-pl = {Lemmas},
	name-pl = {lemmas},
}

\newtheorem{cor}[thm]{Corollary}
\AddToHook{env/cor/begin}{\zcsetup{countertype={thm=cor}}}
\zcRefTypeSetup{cor}{
	Name-sg = {Corollary},
	name-sg = {corollary},
	Name-pl = {Corollarys},
	name-pl = {corollarys},
}

\newtheorem{prop}[thm]{Proposition}
\AddToHook{env/prop/begin}{\zcsetup{countertype={thm=prop}}}
\zcRefTypeSetup{prop}{
	Name-sg = {Proposition},
	name-sg = {proposition},
	Name-pl = {Propositions},
	name-pl = {propositions},
}

\newtheorem{opprob}[thm]{Open~problem}
\AddToHook{env/opprob/begin}{\zcsetup{countertype={thm=open-problem}}}

\newtheorem{problem}[thm]{Problem}
\AddToHook{env/problem/begin}{\zcsetup{countertype={thm=problem}}}
\zcRefTypeSetup{problem}{
	Name-sg = {Problem},
	name-sg = {problem},
	Name-pl = {Problems},
	name-pl = {problems},
}

\theoremstyle{remark}

\newtheorem{rem}[thm]{Remark}
\AddToHook{env/rem/begin}{\zcsetup{countertype={thm=rem}}}
\zcRefTypeSetup{rem}{
	Name-sg = {Remark},
	name-sg = {remark},
	Name-pl = {Remarks},
	name-pl = {remarks},
}

\newtheorem{example}[thm]{Example}
\AddToHook{env/example/begin}{\zcsetup{countertype={thm=example}}}

\newtheorem*{rem*}{Remark}
\AddToHook{env/rem*/begin}{\zcsetup{countertype={thm=remark}}}

\newtheorem*{rems*}{Remarks}
\AddToHook{env/rems*/begin}{\zcsetup{countertype={thm=remarks}}}

\zcRefTypeSetup{section}{
	Name-sg = {\S},
	name-sg = {\S},
	Name-pl = {\S\S},
	name-pl = {\S\S},
	namesep = {\,},
}
\zcRefTypeSetup{equation}{
	Name-sg-ab = {},
	name-sg-ab = {},
	Name-pl-ab = {},
	name-pl-ab = {},
	namesep = {},
	abbrev,
}
\zcRefTypeSetup{item}{
	Name-sg-ab = {},
	name-sg-ab = {},
	Name-pl-ab = {},
	name-pl-ab = {},
	namesep = {},
	abbrev,
}

\hypersetup{colorlinks=false, pdfborder={0 0 0}, pdfborderstyle={/S/U/W 0}}

\newcommand{\I}{\mathrm{i}}
\newcommand{\CC}{\mathbb{C}}
\newcommand{\DD}{\mathbb{D}}
\newcommand{\dd}[1]{\mathop{\mathrm{d}#1}}
\newcommand{\Ev}[2][X]{ \mathop{\_}(#2) }
\DeclarePairedDelimiter\inner{\langle}{\rangle}
\DeclarePairedDelimiter\parentheses{\lparen}{\rparen}
\DeclarePairedDelimiter\braces{\lbrace}{\rbrace}
\DeclarePairedDelimiter\norm{\lVert}{\rVert}
\DeclarePairedDelimiter\abs{\lvert}{\rvert}
\newcommand{\evalbar}{\biggr\rvert}
\NewDocumentCommand\set{ s o m o }{%
	\IfBooleanTF{#1}{\IfNoValueTF{#4}{\braces*{#3}}{\braces*{\,#3:#4\,}}}{%
	\IfNoValueTF{#2}{\IfNoValueTF{#4}{\braces{#3}}{\braces{\,#3:#4\,}}}{%
	\IfNoValueTF{#4}{\braces[#2]{#3}}{\braces[#2]{\,#3:#4\,}}}}%
}
\newcommand{\myTVect}[2][{\boldsymbol{t}}]{ \boxed{\qquad \strut \mathclap{\smash{#2#1^n}} \qquad} }%

\numberwithin{equation}{section}

\newif\ifshowfigures
\showfigurestrue

\usepackage{xifthen,ifthen}
\makeatletter
\newcounter{@ToDo}
\newcommand{\todo@helper}[1]{%
	({\color{blue}TODO~\arabic{@ToDo}: {#1\@addpunct{.}}})%
}
\newcommand{\todo}[1]{\stepcounter{@ToDo}%
	\relax\ifmmode\text{\todo@helper{#1}}%
	\else\todo@helper{#1}\fi%
}
\makeatother

\title[Shapiro's problem]{Shapiro's problem on polynomials with large partial sums of coefficients}
\date{\today{}}


\author{Marc~Technau}
\address{%
	Marc~Technau\\%
	Institut für Analysis und Zahlentheorie\\%
	Graz University of Technology\\%
	Kopernikusgasse~24/II\\%
	8010~Graz\\%
	Austria}
\email{mtechnau@math.tugraz.at}
\thanks{The author is funded by the \emph{Austrian Science Fund~(FWF)}, project \emph{`Prime divisors of polynomials, spin chains, and non-residues'}; grant DOI: \href{http://doi.org/10.55776/PAT4579123}{10.55776/PAT4579123}.}

\begin{document}
\begin{abstract}
	Given a polynomial $\sum_\nu a_\nu X^\nu$ of degree $<d$, bounded by one on the unit disk, how large can $\abs{a_0+a_1+\ldots+a_n}$ ($n<d$) get?

	This question dates back at least to the~1952 thesis work of H.~S.~Shapiro\nocite{shapiro1952}.
	In~1978, D.~J.~Newman\nocite{newman1978} gave an \emph{exact} answer for $d=2(n+1)$, but there does not seem to have been further progress on the question since.
	We study variations on this theme, obtaining exact answers for some related coefficient sums, and answer the original question in an asymptotic sense, provided that $n$ is `not too large' in terms of $d$.
	The latter is achieved via a `quantitative' Eneström--Kakeya\nocite{enestrom1893,kakeya1912} theorem, while the former is based on certain identities for carefully selected Lagrange interpolators.
	
	From the interpolation approach we also obtain a general inequality for coefficient sums $\abs{ t_0 a_0 + \ldots + t_{d-1} a_{d-1} }$ for \emph{arbitrary} complex numbers $t_0,\ldots,t_{d-1}$.
	This inequality fails to be sharp in general, yet it is in some cases and also yields non-trivial bounds for Shapiro's problem for some choices of $n$ and $d$.
\end{abstract}
\maketitle

\section{Introduction}

For any function $f$ on the complex unit disk $\DD$, let $\norm{f}_\infty = \sup\set{ \abs{f(z)} }[ z\in\DD ]$.
In his~1952 PhD thesis, H.~S.~Shapiro~\cite[§\,2.8]{shapiro1952} studied the following problem (\emph{cf}.\ also his Master's thesis~\cite{shapiro1951} with the same title):
\begin{problem}[Shapiro, 1952]\label{prob:Shapiro}
	For integers $d>n\geq 0$, determine the maximum $\mathscr{M}_{n,d}$ of
	\begin{equation}
		\label{eq:FunctionalAbsValue}
		\abs{ a_0+a_1+\ldots+a_n },
	\end{equation}
	over all polynomials $f = a_0 + a_1 X + \ldots + a_{d-1} X^{d-1} \in \CC[X]$ of degree strictly less than $d$ with $\norm{f}_\infty \leq 1$.
\end{problem}
Clearly, \zcref{prob:Shapiro} may be viewed as determining the \emph{operator norm} of a certain linear functional (here the underlying space being $\CC[X]^{<d}$, the space of polynomials of degree $<d$).
This relates the present problem to a vast family of problems in functional analysis and adjacent fields, the scope of which we cannot possibly frame here in a just manner.

More concretely, the motivation for Shapiro's question seems to have been (at least in part) a result of Landau, who had determined the maximum of~\zcref{eq:FunctionalAbsValue} when instead $f$ ranges over \emph{all $1$-bounded holomorphic functions} on the unit disk (see~\cite{landau1913} or~\cite[pp.~20--23]{landau1916}):
\begin{thm}[Landau, 1913]\label{thm:Landau}
	Let $n$ be a non-negative integer.
	Then, for any holomorphic function $f$ on $\DD$ with $\norm{f}_\infty \leq 1$ and power series expansion $f(z) = a_0 + a_1 z + a_2 z^2 + \ldots$, one has
	\begin{equation}
		\label{eq:LandausBound}
		\abs{ a_0+a_1+\ldots+a_n }
		\leq \sum_{\nu=0}^n \binom{-1/2}{\nu}^2 \eqqcolon \mathscr{L}_{n}
		\quad\parentheses*{ = \frac{\log n}{\pi}+O(1) },
	\end{equation}
	where
	\begin{equation}
		\label{eq:BinomialCoefficient}
		\binom{-1/2}{\nu}
		= \prod_{\kappa=0}^{\nu-1} \frac{-1/2-\kappa}{\nu-\kappa}
		= (-1)^\nu \frac{1\cdot 3\cdot \ldots \cdot (2\nu-1)}{2\cdot 4\cdot \ldots \cdot 2\nu}.
	\end{equation}
	Moreover, `$\leq$' in~\zcref{eq:LandausBound} is attained with equality precisely for the rational function $f_n$ given by
	\begin{equation}
		\label{eq:LandausBound:Extremizer}
		f_n(z) = \frac{z^n P_n(1/z)}{P_n(z)},
		\quad\text{where}\quad
		P_n(z) = \sum_{\nu=0}^n \binom{-1/2}{\nu} (-z)^\nu,
	\end{equation}
	and the functions $\eta f_n$, where $\eta$ is any unimodular constant.
\end{thm}
\begin{rem*}
	Landau~\cite{landau1913,landau1916} proves the asymptotic in~\zcref{eq:LandausBound} with an error term $o(\log n)$ rather than $O(1)$.
	The proof is based on Stirling's formula.
	Newman~\cite{newman1978} cites Landau's result with the stronger error term $O(1)$.
	Regardless, the latter version is immediate from Landau's proof upon using a stronger form of Stirling's formula.
\end{rem*}

Let us now restrict attention to those $f$ in \zcref{thm:Landau} which are polynomials of degree $<d$.
By compactness, the set of all possible numbers produced by the left hand side of~\zcref{eq:LandausBound} admits a maximum $\mathscr{M}_{n,d}$ (recall \zcref{prob:Shapiro}).
As the extremal function $f_n$ in \zcref{thm:Landau} is only a polynomial for $n=0$, we have the \emph{strict} inequality
\begin{equation}
	\label{eq:PolHolComparison}
	\mathscr{M}_{n,d}
	< \mathscr{L}_n
	\quad
	(\leq \mathscr{L}_d).
\end{equation}
(For a reasonably self-contained justification of this inequality, see \zcref{prop:LandauNotSharpForPolys} below.)
Despite this, in a stroke of genius, Newman~\cite{newman1978} was able to compute the numbers $\mathscr{M}_{n-1,2n}$ \emph{exactly}:
\begin{thm}[Newman, 1978]\label{thm:Newman}
	Let $n$ be a non-negative integer.
	Then, for every polynomial $f = \sum_\nu a_\nu X^\nu \in\CC[X]^{<2n}$ with $\norm{f}_\infty \leq 1$, one has
	\begin{equation}
		\label{eq:NewmansBound}
		\abs{ a_0+a_1+\ldots+a_{n-1} }
		\leq \frac{1}{2} + \frac{1}{n} \sum_{\omega^n = -1} \frac{1}{\abs{\omega-1}}.
	\end{equation}
	(Here the sum ranges over all complex numbers $\omega$ with $\omega^n=-1$.)
	Moreover, `$\leq$' in~\zcref{eq:NewmansBound} is attained with equality for some polynomial satisfying the above conditions; in other words,
	\begin{equation}
		\label{eq:NewmansBound:Variant}
		\mathscr{M}_{n-1,2n}
		= \frac{1}{2} + \frac{1}{n} \sum_{\omega^n = -1} \frac{1}{\abs{\omega-1}}
		\quad\parentheses*{ = \frac{\log n}{\pi}+O(1) }.
	\end{equation}
\end{thm}

We conclude this section with a short remark on the real case, \emph{i.e.}, inequalities for coefficient sums of real polynomials that are bounded by~$1$ on the interval $[-1,1]$.
In~1892, Markov established sharp bounds for the derivatives of such polynomials (see~\cite[p.~258]{markov1916}).
The question analogous to~\zcref{eq:FunctionalAbsValue} seems to have first been considered by Reimer and Zeller~\cite{reimer1967} who established sharp bounds for even/odd polynomials (see also~\cite{rack1989} and the references therein).
As in the above complex setting, the current state of knowledge concerning the real case seems to be similarly fragmented and leaves much to be desired.
Regardless, the methods that have been used to study the real situation have a somewhat different flavour, and in the present investigation we shall restrict our attention to the complex case.

\section{Results}

\subsection{Approximate formul\ae{}}

If $n$ is kept fixed, then it follows from a (more general) result of Shapiro~\cite[Theorem~13 in §\,2.7]{shapiro1952} that $\mathscr{M}_{n,d} = \mathscr{L}_{d} - O_n(\epsilon_n^d)$ as $d\to\infty$.
(Here $\epsilon_n$ is some positive number $<1$.)
Our first result provides an approximate lower bound for $\mathscr{M}_{n,d}$ when \emph{both} parameters $n<d$ grow:
\begin{thm}\label{thm:Quantitative}
	Let $d>n$ be a positive integers.
	Then the exists a polynomial $\tilde{f}_{n,d} = \sum_\nu a_\nu X^\nu \in\CC[X]^{<d}$ with $\norm{\tilde{f}_{n,d}}_\infty \leq 1$ such that
	\[
		\abs{ a_0+a_1+\ldots+a_n }
		= \parentheses*{ 1 - O\braces*{ n^3 \exp\parentheses*{-\frac{d}{5n}} } } \sum_{\nu=0}^n \binom{-1/2}{\nu}^2.
	\]
\end{thm}
On combining this with \zcref{thm:Landau} we immediately deduce the following result:
\begin{cor}\label{cor:MndGrowth}
	\(\displaystyle
		\mathscr{M}_{n,d} = \frac{\log n}{\pi} + O(1)
	\)
	for $d > 16 \mkern 2mu n \log n$ and $n\to\infty$.
\end{cor}

Newman~\cite[p.~187]{newman1978} also remarks that by~\zcref{eq:PolHolComparison} and~\zcref{eq:LandausBound} one has
\[
	\mathscr{M}_{n,d} \leq \frac{\log d}{\pi} + O(1),
\]
and proceeds to ask whether \emph{`[the left hand side] ever gets this large'} (as $n$ varies from $0$ to $d-1$).
He points out that~\emph{\zcref{eq:NewmansBound:Variant} answers this in the affirmative}.
We remark that our \zcref{cor:MndGrowth} furnishes the following estimate:
\begin{cor}
	\(\displaystyle
		\mathscr{M}_{n,d} = \frac{\log d}{\pi} + O(\log\log d)
	\)
	provided that
	\(\displaystyle
		n \asymp \frac{d}{\log d}
	\).
\end{cor}
This may be regarded as \emph{another positive answer to Newman's question} up to lower order terms.

\begin{rem}[Improvements to the error terms]
	We have refrained from attempting to get the sharpest error term in \zcref{thm:Quantitative} that our method allows for.
	Also the number $16 \;(>3\cdot 5)$ in \zcref{cor:MndGrowth} is not optimal in this regard.
	It would be particularly pleasing if \zcref{cor:MndGrowth} could be established for $d\gg n$, but we do not know if this can be achieved.
	We refer to~\zcref{subsec:RemarksOnZeros} below for further discussion pertaining to potential improvements.
\end{rem}

\subsection{Exact formul\ae{}}
\label{subsec:ShapiroExact}

Before stating further results, let us provide more context.
Shapiro~\cite{shapiro1952} also considers extremal problems for more general sums than in~\zcref{eq:FunctionalAbsValue}, namely the question of maximising
\begin{equation}\label{eq:FunctionalAbsValue:General}
	\abs{ t_0 a_0 + \ldots + t_{d-1} a_{d-1} }
\end{equation}
over all polynomials $f = \sum_\nu a_\nu X^\nu \in\CC[X]^{<d}$ with $\norm{f}_\infty \leq 1$, where $\boldsymbol{t} = (t_0,\ldots,t_{d-1})\in\CC^d$ is fixed.
The outcome of his theory is a better theoretical understanding of the behaviour of the extremal polynomials (for the above maximisation problem) on the unit circle.
In particular, for those $\boldsymbol{t}$ for which the extremal polynomials are not given by monomials, he finds that for some $r\leq d$ there are complex numbers $z_1,\ldots,z_r, u_1,\ldots,u_r$ with $\abs{z_1} = \ldots = \abs{z_r} = 1$ such that
\begin{equation}\label{eq:FunctionalInterpolation}
	t_0 a_0 + \ldots + t_{d-1} a_{d-1}
	= \sum_{\nu=1}^r u_\nu f(z_\nu)
\end{equation}
for all $f$ as above, and the the answer to the maximisation of~\zcref{eq:FunctionalAbsValue:General} is given \emph{exactly} by $\abs{u_1}+\ldots+\abs{u_r}$ (see~\cite[§§\,2.3--4]{shapiro1952}).

In Shapiro's terminology, the $z_\nu$ are called the \emph{nodes} associated to the functional given by $\boldsymbol{t}$.
Observe that~\zcref{eq:FunctionalInterpolation} may be interpreted as decomposing said functional using point evaluations at the $z_\nu$ (see also~\zcref{eq:FunctionalPointEvalDecomp} below), yet there may be many more such decompositions with point evaluations not at the nodes.
Despite this promising insight, Shapiro notes that \emph{`[he] has been unable to devise any method for finding the nodes'}~(\cite[§\,2.4\,2]{shapiro1952}).\footnote{%
	Note, however, the partial progress in~\cite[§\,2.8]{shapiro1952} for some particular examples.
}

Here we point out that Newman's method~\cite{newman1978} furnishes a (thin) host of examples for which the nodes \emph{can} be computed.
The next theorem (proved in~\zcref{subsec:ApplicationToFunctionals:CosmeticSurgery} below) follows rather easily from a careful reading of Newman's argument:\footnote{%
	Newman does not state this as a result, yet it would be hard to conceive that he would not have been aware of this.
}
\begin{thm}\label{thm:Newman:Gen}
	Let $n$ and $d$ be positive integers such that $d\geq 2n-1$.
	Moreover, let
	\[
		\boldsymbol{t}^n = (t_{0\bmod 2n},\ldots,t_{n-1\bmod 2n}) \in \CC^n
	\]
	be arbitrary.
	Let $t_{n+\nu\bmod 2n} = - t_{\nu\bmod 2n}$ for $\nu=0,\ldots,n-1$.
	Then, for every polynomial $f = \sum_\nu a_\nu X^\nu \in\CC[X]^{<d}$ with $\norm{f}_\infty \leq 1$, one has
	\begin{equation}
		\label{eq:NewmanGeneralised:Plus}
		\abs[\bigg]{ \sum_{\nu=0}^{d-1} t_{\nu\bmod2n} \mkern 2mu a_\nu }
		\leq \frac{1}{n} \sum_{\omega^n=-1} \abs[\bigg]{ \sum_{\nu=0}^{n-1} \frac{t_{\nu\bmod2n}}{\omega^\nu} }.
	\end{equation}
	Moreover, there is a choice of $f$ as above for which `$\leq$' in~\zcref{eq:NewmanGeneralised:Plus} is attained with equality.
	In fact, this $f$ may even be chosen to have degree $\leq 2n-2$.
\end{thm}

\begin{rem*}
	Evidently, \zcref{thm:Newman:Gen} falls short of implying \zcref{thm:Newman}.
	Here the latter theorem corresponds to the choice $\boldsymbol{t} = (1,\ldots,1,0,\ldots,0) \in \CC^d$ in~\zcref{eq:FunctionalAbsValue:General}, with equally many $1$'s and $0$'s.
	However, \zcref{thm:Newman:Gen} does handle $\boldsymbol{t} = (1,\ldots,1,-1,\ldots,-1)$ and Newman~\cite{newman1978} provides a rather ingeneous argument for bridging this gap.
\end{rem*}

We note that the vectors $\boldsymbol{t}$ in~\zcref{eq:FunctionalAbsValue:General} to which \zcref{thm:Newman:Gen} applies are of the form
\[
	\boldsymbol{t} = \parentheses*{\,
		\myTVect{ }  \mkern 2mu, \,
		\myTVect{-}  \mkern 2mu, \,
		\ldots       \mkern 2mu, \,
		\myTVect{\pm}            \,
		\mathllap{ /\!/\!/ \,\, }
	}
	\in \CC^d,
\]
where `$/\!/\!/$' indicates that suitably many last few entries in the final block are to be omitted as to have $\boldsymbol{t}$ belong to $\CC^d$.
We are also able to prove the following variant of \zcref{thm:Newman:Gen} which handles the more symmetric case of vectors $\boldsymbol{t}$ of the form
\[
	\boldsymbol{t} = \parentheses*{\,
		\myTVect{ }  \mkern 2mu, \,
		\myTVect{ }  \mkern 2mu, \,
		\ldots       \mkern 2mu, \,
		\myTVect{ }              \,
		\mathllap{ /\!/\!/ \,\, }
	}
	\in \CC^d:
\]

\begin{thm}\label{thm:Newman:Gen:Variant}
	Let $n$ and $d$ be positive integers such that $d\geq 2n-1$.
	Moreover, let
	\[
		\boldsymbol{t}^n = (t_{0\bmod n},\ldots,t_{n-1\bmod n}) \in \CC^n
	\]
	be arbitrary.
	Then, for every polynomial $f = \sum_\nu a_\nu X^\nu \in\CC[X]^{<d}$ with $\norm{f}_\infty \leq 1$, one has
	\begin{equation}
		\label{eq:NewmanGeneralised:Minus}
		\abs[\bigg]{ \sum_{\nu=0}^{d-1} t_{\nu\bmod n} \mkern 2mu a_\nu }
		\leq \frac{1}{n} \sum_{\omega^n=1} \abs[\bigg]{ \sum_{\nu=0}^{n-1} \frac{t_{\nu\bmod n}}{\omega^\nu} }.
	\end{equation}
	Moreover, there is a choice of $f$ as above for which `$\leq$' in~\zcref{eq:NewmanGeneralised:Minus} is attained with equality.
	In fact, this $f$ may even be chosen to have degree $\leq 2n-2$.
\end{thm}

As an application of (either of) the two theorems we prove the following cute inequality:

\begin{cor}\label{cor:Application:CuteIneq}
	Let $0\leq k<n$ be integers.
	Then, for any polynomial $f\in\CC[X]^{<2n+k}$,
	\begin{equation}\label{eq:Application:CuteIneq}
		\abs*{\frac{f^{(k)}(0)}{k!}} + \abs*{\frac{f^{(n+k)}(0)}{(n+k)!}}
		\leq \norm{f}_\infty.
	\end{equation}
	Moreover, equality occurs for some polynomial of degree $\leq 2n-2$.
\end{cor}

\subsection{A general bound}
\label{subsec:GeneralBound}

In both \zcref{thm:Newman:Gen} and \zcref{thm:Newman:Gen:Variant} we are able to handle an $n$-dimensional subspace of functionals on $\CC[X]^{<d}$.
Specifically for $d\in\set{2n-1,2n}$ the sum of these subspaces turns out to be the \emph{whole} dual space of $\CC[X]^{<d}$.
The underlying reason is that (in the notation from~\zcref{subsec:ShapiroExact}) every vector $\boldsymbol{t}\in\CC^d$ admits a decomposition of the form
\[
	\boldsymbol{t}
	= \parentheses*{\,
		\myTVect[{\boldsymbol{t}_{+}}]{ } \mkern 2mu, \,
		\myTVect[{\boldsymbol{t}_{+}}]{ }             \,
		\mathllap{ /\!/\!/ \,\, }
	}
	+ \parentheses*{\,
		\myTVect[{\boldsymbol{t}_{-}}]{ } \mkern 2mu, \,
		\myTVect[{\boldsymbol{t}_{-}}]{-}             \,
		\mathllap{ /\!/\!/ \,\, }
	}
\]
for suitable vectors $\boldsymbol{t}_{+}^n, \boldsymbol{t}_{-}^n \in \CC^n$ (see~\zcref{eq:NewmanGeneralised:Combined:Proof:tVectDecomp} below).

In general, this does not, unfortunately, enable us to obtain \emph{sharp} bounds for~\zcref{eq:FunctionalAbsValue:General}.
However, we still get upper bounds:

\begin{cor}\label{cor:NewmanGeneralised:Combined}
	Let $d$ be a positive integer and write $D = d$ or $=d-1$ according to whether $d$ is even or odd.
	Let $\boldsymbol{t} = (t_0,\ldots,t_{d-1}) \in \CC^d$ be arbitrary.
	Then, for every polynomial $f = \sum_\nu a_\nu X^\nu \in\CC[X]^{<d}$ with $\norm{f}_\infty \leq 1$, one has
	\begin{equation}
		\label{eq:NewmanGeneralised:Combined}
		\abs{ t_0 a_0 + \ldots + t_{d-1} a_{d-1} }
		\leq \frac{1}{D} \sum_{\omega^D=1} \abs[\bigg]{ \sum_{\nu=0}^{d-1} \frac{t_\nu}{\omega^\nu} }.
	\end{equation}
\end{cor}

Lastly, let us use Shapiro's example~\zcref{eq:FunctionalAbsValue} as a benchmark for \zcref{cor:NewmanGeneralised:Combined}.

\begin{rem}\label{rem:ComparisonAgainstNewmansResult}
	For $d=2n$ and $\boldsymbol{t} = (1_{\times n},0_{\times n}) \in \CC^d$ the right hand side~\zcref{eq:NewmanGeneralised:Combined} turns out to be precisely Newman's bound~\zcref{eq:NewmansBound}.
	However, we cannot quite deduce Newman's result (\zcref{thm:Newman}) from \zcref{cor:NewmanGeneralised:Combined}, because our bound~\zcref{eq:NewmanGeneralised:Combined} is not guaranteed to be sharp.
	The fact that the bound~\zcref{eq:NewmanGeneralised:Combined} may indeed \emph{fail} to be sharp can be seen from the next example (namely the fact that there are white parts in \zcref{fig:UpperBoundQualityTest}).
\end{rem}

\begin{figure}[!ht]%
	\centering%
	\begingroup%
	\def\magnif{4}%
	\colorlet{colorSpy}{black}%
	\colorlet{colorBg}{black!10}%
	\colorlet{colorKnown}{black!20}%
	\colorlet{colorNew}{black!45}%
	\begin{tikzpicture}[
			scale=.1,
			spy using outlines={ circle, magnification=\magnif, size=1cm, connect spies }
		]
		\coordinate (O) at (.5,.5);
		\pgfmathsetmacro{\m}{7}
		\fill[colorBg, opacity=.5] (O) ++(-.5,-.5)  -- ++(.5,0) -- ++(0,.5) \foreach \x in {1,2,...,\m}{ -- ++(1,0) -- ++(0,1) } -- ++(.5,0) -- ++(0,.5) -- ++(-\m-1,0) -- cycle;
		\foreach \x/\y in {%
			2/1, 3/1, 4/1, 4/3, 5/1, 5/2, 6/1, 6/2, 6/5, 7/3, 8/3, 8/7, 9/4, 10/4, 10/9, 11/5, 12/5, 12/11, 13/6, 13/12, 14/6, 14/12, 14/13, 15/7, 15/14, 16/7, 16/14, 16/15, 17/8, 17/16, 18/8, 18/16, 18/17, 19/9, 19/18, 20/9, 20/18, 20/19, 21/10, 21/20, 22/10, 22/20, 22/21, 23/11, 23/22, 24/11, 24/22, 24/23, 25/12, 25/24, 26/12, 26/24, 26/25, 27/13, 27/26, 28/13, 28/26, 28/27, 29/14, 29/28, 30/14, 30/28, 30/29, 31/15, 31/30, 32/15, 32/30, 32/31, 33/16, 33/31, 33/32, 34/16, 34/31, 34/32, 34/33, 35/17, 35/33, 35/34, 36/17, 36/33, 36/34, 36/35, 37/18, 37/35, 37/36, 38/18, 38/35, 38/36, 38/37, 39/19, 39/37, 39/38, 40/19, 40/37, 40/38, 40/39, 41/20, 41/39, 41/40, 42/20, 42/39, 42/40, 42/41, 43/21, 43/41, 43/42, 44/21, 44/41, 44/42, 44/43, 45/22, 45/43, 45/44, 46/22, 46/43, 46/44, 46/45, 47/23, 47/45, 47/46, 48/23, 48/45, 48/46, 48/47, 49/24, 49/47, 49/48, 50/24, 50/47, 50/48, 50/49, 51/25, 51/49, 51/50, 52/25, 52/49, 52/50, 52/51, 53/26, 53/51, 53/52, 54/26, 54/51, 54/52, 54/53, 55/27, 55/53, 55/54, 56/27, 56/53, 56/54, 56/55, 57/28, 57/54, 57/55, 57/56, 58/28, 58/54, 58/55, 58/56, 58/57, 59/29, 59/56, 59/57, 59/58, 60/29, 60/56, 60/57, 60/58, 60/59, 61/30, 61/58, 61/59, 61/60, 62/30, 62/58, 62/59, 62/60, 62/61, 63/31, 63/60, 63/61, 63/62, 64/31, 64/60, 64/61, 64/62, 64/63, 65/32, 65/62, 65/63, 65/64, 66/32, 66/62, 66/63, 66/64, 66/65, 67/33, 67/64, 67/65, 67/66, 68/33, 68/64, 68/65, 68/66, 68/67, 69/34, 69/66, 69/67, 69/68, 70/34, 70/66, 70/67, 70/68, 70/69, 71/35, 71/68, 71/69, 71/70, 72/35, 72/68, 72/69, 72/70, 72/71, 73/36, 73/70, 73/71, 73/72, 74/36, 74/70, 74/71, 74/72, 74/73, 75/37, 75/72, 75/73, 75/74%
		}{
			\coordinate (x) at ($(\x,\y)-(O)$);
			\pgfmathsetmacro{\mx}{int(max(\x,\m))}
			\global\let\m=\mx
			\pgfmathparse{ifthenelse(2*\y+2 == \x || \y+1 == \x, 1, 0)}
			\ifodd\pgfmathresult\relax
				\draw[fill=colorKnown]
					(x) rectangle ++(1,1)
				;
			\else
				\draw[fill=colorNew] (x) rectangle ++(1,1);
			\fi
		}
		\begin{pgfonlayer}{background}
			\fill[colorBg] (O) ++(-.5,-.5)  -- ++(.5,0) -- ++(0,.5) \foreach \x in {1,2,...,\m}{ -- ++(1,0) -- ++(0,1) } -- ++(.5,0) -- ++(0,.5) -- ++(-\m-1,0) -- cycle;
			\draw[step=10, ultra thin] (0,0) grid ++(\m+1,\m+1);
			\pgfmathsetmacro{\n}{20}
			\draw[<-, >=latex, shorten <=2pt] ({2*\n},{\n-1}) -- ++(-30:8) node[font=\footnotesize, below right, inner sep=0pt, fill=white] {\( (2n,n-1) \)};
		\end{pgfonlayer}
		\draw (0,0) rectangle ++(\m+1,\m+1);
		\foreach \x [evaluate=\x as \l using { mod(\x,10)==0 ? 0 : 1 }] in {1,2,...,\m}{
			\coordinate (X\x) at (\x,0);
			\coordinate (x\x) at (\x,\m+1);
			\coordinate (Y\x) at (\m+1,\x);
			\coordinate (y\x) at (0,\x);
			\draw[shift only, very thin, line cap=round]
				(X\x) -- ++(0,{ \l*2pt})
				(Y\x) -- ++(0:{-\l*2pt})
				(x\x) -- ++(0,{-\l*2pt})
				(y\x) -- ++(0:{ \l*2pt})
			;
		}
		\foreach \x [evaluate=\x as \l using {1.8}] in {10,20,...,\m}{
			\draw[shift only, line cap=round]
				(X\x) -- ++(0,{ \l*2pt})
				(Y\x) -- ++(0:{-\l*2pt})
				(x\x) -- ++(0,{-\l*2pt})
				(y\x) -- ++(0:{ \l*2pt})
			;
			\node[font=\tiny, below, inner sep=0, yshift=-4pt] at (X\x) {\(\x\)};
			\node[font=\tiny, left , overlay] at (y\x) {\(\x\)};
		}
		\coordinate (S) at (3,2);
		\coordinate (R) at (20,55);
		\draw[colorBg, shorten <={\magnif*1mm}, line width=5pt] (S) -- (R);
		\fill[colorBg] (R) circle ({\magnif*4cm+10mm});
		\spy[size=32mm, colorSpy] on (S) in node[fill=white] at (R);
		\begin{pgfonlayer}{fg}
			\begin{scope}
				\clip (R) circle ({\magnif*4cm});
				\begin{scope}[shift={($(R)-\magnif*(S)$)}, scale=\magnif, every node/.style={ inner sep=1pt, font=\tiny, right }]
					\node (N) at (0.2, 2.4) {Ineq.\,\zcref{eq:NumericalExperiment:1}};
					\node (n) at (2.8, 4) {Ineq.\,\zcref{eq:NumericalExperiment:2}};
					\path[->, >=latex, shorten >=.5pt, black]
						(N) edge (3,1) (N) edge (5,1) (N) edge (6,1)
						(5,2.5) coordinate (r)
						(r|-n.south) edge (r)
					;
					\foreach \x in {1,2,...,6}{
						\node[font=\tiny, below, inner sep=0, yshift=-4pt] at (\x,0) {\(\x\)};
						\node[font=\tiny, left , overlay, xshift=-2pt] at (0,\x) {\(\x\)};
					}
				\end{scope}
			\end{scope}
		\end{pgfonlayer}
	\end{tikzpicture}%
	\endgroup%
	\caption[]{\label{fig:UpperBoundQualityTest}%
		A numerical experiment showing when \zcref{cor:NewmanGeneralised:Combined} produces stronger upper bounds on the maximum $\mathscr{M}_{n,d}$ from \zcref{prob:Shapiro} than previously known (cf.~\zcref{example:NumericalExperiment}).\\
		The plot shows the points $(d,n)$, $0\leq n<d\leq75$, for which the upper bound~\zcref{eq:NewmanGeneralised:Combined} furnished by \zcref{cor:NewmanGeneralised:Combined} for $\boldsymbol{t} = (1_{\times (n+1)},0_{\times(d-n-1)}) \in \CC^d$ is strictly smaller than the bound $\mathscr{L}_n$ for $\mathscr{M}_{n,d}$ coming from Landau's theorem (see~\zcref{eq:PolHolComparison})%
	}%
\end{figure}

\begin{example}
	\label{example:NumericalExperiment}
	Using a computer, we have investigated when \zcref{cor:NewmanGeneralised:Combined} yields non-trivial information about~\zcref{eq:PolHolComparison}:
	for $0\leq n<d\leq 75$ we have computed the upper bound~\zcref{eq:NewmanGeneralised:Combined} ($\mathscr{C}_{n,d}$ say) furnished by \zcref{cor:NewmanGeneralised:Combined} for $\boldsymbol{t} = (1_{\times (n+1)},0_{\times(d-n-1)}) \in \CC^d$.
	For instance, we obtain the following (hitherto unknown) upper bounds for $\mathscr{M}_{1,d}$ ($d=3,5,6$):
	\begin{equation}\label{eq:NumericalExperiment:1}
		\left.
		\begin{array}{ @{} r @{}c@{} l @{}c@{} c@{}c@{} l }
			\mathscr{M}_{1,3}
			&{}\leq{}& \mathscr{C}_{1,3}
			&{}  = {}& \smash{ \tfrac{1}{2} + \tfrac{1}{\sqrt{2}} }
			&{}  = {}& 1.2071\ldots \\[3pt]
			\mathscr{M}_{1,5}
			&{}\leq{}& \mathscr{C}_{1,5}
			&{}  = {}& \smash{ \tfrac{2}{3} + \tfrac{1}{\sqrt{3}} }
			&{}  = {}& 1.2440\ldots \\[3pt]
			\mathscr{M}_{1,6}
			&{}\leq{}& \mathscr{C}_{1,6}
			&{}  = {}& \smash{ \tfrac{2}{3} + \tfrac{1}{\sqrt{3}} }
		\end{array}
		\right\rbrace
		< 1.25
		= \tfrac{5}{4}
		= \mathscr{L}_1.
	\end{equation}
	Furthermore, we have
	\begin{equation}\label{eq:NumericalExperiment:2}
		\mathscr{M}_{2,5}
		\leq \mathscr{C}_{2,5}
		= \tfrac{4}{3}
		= 1.3\dot{3}
		< 1.390625
		= \tfrac{89}{64}
		= \mathscr{L}_2.
	\end{equation}
	\zcref{fig:UpperBoundQualityTest} shows the tuples $(d,n)$ for which $\mathscr{C}_{n,d} < \mathscr{L}_n$.
	Note that we have the trivial cases $\mathscr{M}_{0,d} = \mathscr{M}_{d-1,d} = 1$, and $\mathscr{M}_{n-1,2n}$ is already covered by Newman's result (recall \zcref{rem:ComparisonAgainstNewmansResult}, though).
	In the figure we have greyed these out.
\end{example}

\section{Plan of the paper}

In~\zcref{sec:PolysAndLandau} we show that Landau's bound~\zcref{eq:LandausBound} is not sharp for polynomials.
In~\zcref{sec:AsymptoticEstimates} we prove \zcref{thm:Quantitative}.
The key result in our approach turns out to be a `quantitative Eneström--Kakeya theorem', stated as \zcref{lem:EnestromKakeya:Quantitative} below.
In~\zcref{sec:Lagrange} we extend Newman's approach.
The key outcome of this is given as \zcref{cor:FunctionalNormComputation} below.
In~\zcref{subsec:ApplicationToFunctionals:CosmeticSurgery}, we use that result to prove \zcref{thm:Newman:Gen} and \zcref{thm:Newman:Gen:Variant}.
Finally, in~\zcref{subsec:ProofOfCorollaries} we prove the two corollaries enunciated in~\zcref{subsec:ShapiroExact} and~\zcref{subsec:GeneralBound}.

Let us also mention a few points about \emph{notation}:
we use both the Landau symbol $O(\,\cdot\,)$ and Vinogradov's notation $\ll$ in the usual fashion.
The implicit constants are always absolute, unless the contrary is explicitly indicated using subscripts (see, \emph{e.g.},~\zcref{eq:LandauExtremizer:PowerSeriesExpansion:CoeffGrowth}).

~~~~~~~~~~~~~~~~~~~~~~~~~~~~~~~~~~~~~~~~~~~~~~~~~~~~~~~
\section{Polynomials and Landau's bound}
\label{sec:PolysAndLandau}

\subsection{Landau's argument}

In this section we give a self-contained proof that Landau's bound~\zcref{eq:LandausBound} is not sharp for polynomials.
Of course, this result is not new.
In fact, it follows directly from Landau's characterisation of the extremal functions in \zcref{thm:Landau} (see~\cite{landau1913}, but not~\cite{landau1916}, where the required characterisation is omitted).
Nevertheless, we believe that revisiting Landau's argument is instructive.
We return to this point in~\zcref{subsec:ReflectionsOnLandau} below.

\begin{prop}\label{prop:LandauNotSharpForPolys}
	Let $d>n\geq 1$ be integers.
	Then there is some $\delta(n,d)>0$ such that for every polynomial $f = \sum_\nu a_\nu X^\nu \in\CC[X]^{<d}$ with $\norm{f}_\infty \leq 1$, one has
	\begin{equation}
		\label{eq:LandauNotSharpForPolys}
		\abs{ a_0+a_1+\ldots+a_n }
		\leq \sum_{\nu=0}^n \binom{-1/2}{\nu}^2 \mkern-3mu - \delta(n,d).
	\end{equation}
\end{prop}

\begin{proof}
	By compactness of the unit ball in $(\CC[X]^{<d}, \norm{\,\cdot\,}_\infty)$, there exists a polynomial $f = \sum_\nu a_\nu X^\nu \in\CC[X]^{<d}$ with $\norm{f}_\infty \leq 1$ for which~\zcref{eq:LandausBound} is maximal.
	Let $f$ be such a polynomial.
	Furthermore, let $P_n$ be given as in~\zcref{eq:LandausBound:Extremizer}.
	Then, by means of comparison with the binomial series expansion
	\[
		\parentheses[\bigg]{ \sum_{\nu=0}^\infty \binom{-1/2}{\nu} (-z)^\nu }^2
		= \parentheses[\bigg]{ \frac{1}{\sqrt{1-z}} }^2
		= \frac{1}{1-z}
		= \sum_{\nu=0}^\infty z^\nu
		\quad(\text{for } z\in\CC,\,\abs{z}<1),
	\]
	it follows that
	\[
		P_n^2(z) = 1 + z + z^2 + \ldots + z^n + \text{terms of higher order}
		\quad(\text{for } z\in\CC),
	\]
	so that Cauchy's formula, trivial estimation and Parseval's identity yield
	\begin{align}
		\label{eq:LandauNotSharpForPolys:Proof}
		\abs{ a_0 + a_1 + \ldots + a_n }
		= \abs[\bigg]{ {\frac{1}{2\pi\I} \oint_{\abs{z}=1} \frac{f(z)}{z^{n+1}} P_n^2(z) \dd{z}} } &
		\leq {\int_0^1 \abs{P_n(\exp(2\pi \I \mkern1mu t))}^2 \dd{t}} \\ &
		\nonumber
		= \text{right hand side of~\zcref{eq:LandausBound}}.
	\end{align}
	Therefore, in order to establish the existence of $\delta(n,d)>0$ as in~\zcref{eq:LandauNotSharpForPolys}, it suffices to observe that `$\leq$' in~\zcref{eq:LandauNotSharpForPolys:Proof} actually holds with `$<$'.
	This is simple, for if we had equality in~\zcref{eq:LandauNotSharpForPolys:Proof} then $\abs{f(z)} = 1$ for all $\abs{z}=1$ (in the first place, at least at the points $z$, $\abs{z}=1$, where $P_n$ does not vanish,\footnote{In fact, $P_n$ is non-vanishing throughout the closed unit disk; see \zcref{cor:LandauExtremizer:PoleBound} below.} but then for all $\abs{z}=1$ by reason of continuity).
	However, this can only happen if $f$ is a monomal (see \zcref{lem:PolyAbsOneOnUnitCircle} below); yet in this case $\abs{ a_0 + a_1 + \ldots + a_n } = \norm{f} = 1$, but the right hand side of~\zcref{eq:LandausBound} is $>1$.
\end{proof}

\begin{rem}
	It follows from a general result of Shapiro, \cite[Theorem~13 in~§\,2.7]{shapiro1952}, that $\delta(n,d)\to 0$ for fixed $n$ and $d\to\infty$.
	Our \zcref{thm:Quantitative} relaxes the condition on $n$.
\end{rem}

The following result (used in the proof of \zcref{prop:LandauNotSharpForPolys}) is a special case of the well-known classification of rational functions of modulus one on the unit circle (see, \emph{e.g.},~\cite[Ch.~14, Ex.~3 on p.~293]{rudin1987}).
We include a proof (without any claim of novelty) both for completeness' sake and for lack of a suitable reference.
\begin{lem}\label{lem:PolyAbsOneOnUnitCircle}
	Let $f\in\CC[X]$ be a polynomial such that $\abs{f(z)}=1$ for all $z\in\CC$ with $\abs{z}=1$.
	Then $f$ is a monomial.
\end{lem}
\begin{proof}
	Let $\bar{f}$ be the polynomial obtained from $f$ by replacing each coefficient by its complex conjugate.
	For $\abs{z}=1$ we have $\overline{z} = 1/z$.
	Consequently, the rational function $\bar{f}(1/X) f(X)$ equals one: (a)~on the unit circle and, \emph{a~fortiori}, (b)~everywhere (by the identity principle, for instance).
	Now if $f$ were to have a non-zero root $z\in\CC$, then we reap the contradiction $1 = \bar{f}(1/z) f(z) = 0$.
	Hence, the zero locus of $f$ is either~$\emptyset$ or~$\set{0}$.
	In both cases, it follows that $f$ is a monomial.
\end{proof}

\subsection{Reflections on Landau's argument}
\label{subsec:ReflectionsOnLandau}

The key to \zcref{prop:LandauNotSharpForPolys} may be described as resting on two observations:
\begin{enumerate}
	\item The expression in question, $a_0+a_1+\ldots+a_n$, admits a nice integral representation, namely
	\begin{equation}\label{eq:CauchyApplied}
		a_0+a_1+\ldots+a_n
		= {\frac{1}{2\pi\I} \oint_{\abs{z}=1} \frac{f(z)}{z^{n+1}} \parentheses[\big]{ 1 + z + z^2 + \ldots + z^n } \dd{z}}.
	\end{equation}
	\item The `integration kernel' $1 + z + z^2 + \ldots + z^n$ in~\zcref{eq:CauchyApplied} can be replaced by $P_n^2(z)$ (recall~\zcref{eq:LandauNotSharpForPolys:Proof}).
	In the holomorphic case (\zcref{thm:Landau}) one can, in fact, use this to get \emph{sharp} bounds via the functions $f_n$ given in~\zcref{eq:LandausBound:Extremizer}.
\end{enumerate}
It is clear that this approach does not readily generalise to determine the maximum $\mathscr{M}_{n,d}$ of the left hand side of~\zcref{eq:LandausBound} as $f$ ranges over all $1$-bounded \emph{polynomials} of degree $<d$: what we lack is sufficient understanding of the precise boundary behaviour of the corresponding extremal polynomials $f$.

Despite not having the required understanding of $f(z)$ for \emph{all} $\abs{z}=1$, it is possible to get control on a finite number of such $z$ using an interpolation-based approach.
In this way one can get results when the integral formula~\zcref{eq:CauchyApplied} can be replaced by a finite sum (involving point evaluations of $f$) with `especially few' terms.
Such phenomena of replacing integration of polynomials by finite sums with few terms is a well-known theme in numerical analysis, namely \emph{Gauss quadrature}.
In our case, this discrete analogue of~\zcref{eq:CauchyApplied} is given in the form of~\zcref{eq:FunctionalPointEvalDecomp} in \zcref{cor:FunctionalNormComputation} below.
We return to this in~\zcref{sec:Lagrange}, when we discuss the exact formul\ae{} presented in~\zcref{subsec:ShapiroExact}.

\section{Asymptotic estimates}
\label{sec:AsymptoticEstimates}

\subsection{Proof of \texorpdfstring{\zcref{thm:Quantitative}}{Theorem~\autoref{thm:Quantitative}}}
\label{sec:thm:Quantitative:Proof}

The strategy is to approximate Landau's extremal function $f_n$ by polynomials; this natural idea is already found in~\cite[Theorem~13 in~§\,2.7]{shapiro1952}, where it is applied to a more general extremal problems.
In the generality of the cited result, one cannot hope to achieve error bounds as explicit as the ones we obtain here.
Our approach hinges upon bounding away the poles of $f_n$ from the unit disk\ifshowfigures{} (\emph{cf}.~\zcref{fig:LandauExtremizer:PoleBound})\fi.
To achieve this, we shall employ the well-known \emph{Eneström--Kakeya theorem}~\cite{enestrom1893,kakeya1912} (see also~\cite[p.~13]{borwein1995}):
\begin{lem}[{Eneström, 1893; Kakeya, 1912}]
	\label{lem:EnestromKakeya}
	Let $a_n X^n + \ldots + a_1 X + a_0$ be a polynomial with positive coefficients $a_0 \geq a_1 \geq \ldots \geq a_n > 0$.
	Then all its complex zeros are contained in the (closed) annulus with radii $\min Q$ and $\max Q$; here
	\(
		Q = \set{ a_\nu / a_{\nu+1} }[ 0\leq \nu<n ]
	\).
\end{lem}

\ifshowfigures%
\begingroup%
\newcommand{\drawFigure}[3][1.15]{%
	\begin{tikzpicture}[scale=#1] 
		\pgfmathsetmacro{\n}{#2}
		\pgfmathsetmacro{\w}{4*#1}
		\coordinate (O) at (0,0);
		\coordinate (A) at ( 1.9,-1.6);
		\coordinate (B) at (-1.9, 1.6);
		\begin{scope}
			\clip (A) rectangle (B);
			\node at (0,0) {\includegraphics[width=\w cm]{plot-f#2.jpg}};
			\fill[even odd rule, white, opacity=.1] (O) circle (1) circle (3);
			\draw (O) circle (1);
			\draw[dashed] (O) circle ({1+1/(2*\n+1)});
			\foreach \a/\r in {#3}{
				\coordinate (P) at ({\a}:{1*\r});
				\coordinate (Z) at ({\a}:{1/\r});
				\draw[shift only, fill=black] (Z) circle (2pt);
				\draw[shift only, fill=white] (P) circle (2pt);
			}
		\end{scope}
		\draw (A) rectangle (B);
		\node[below right, font=\small] at (B) {\( n = #2 \):};
	\end{tikzpicture}
}%
\begin{figure*}[!ht]%
	\centering%
	\drawFigure{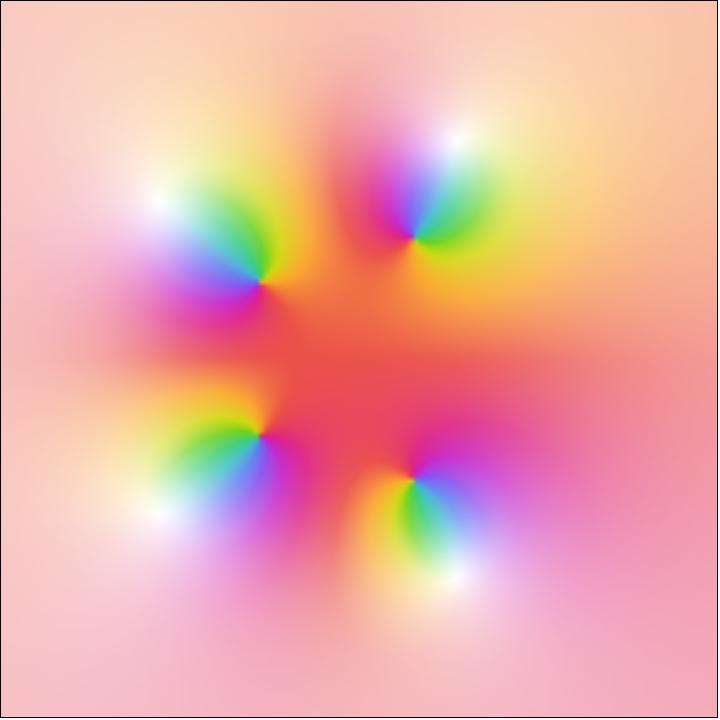}{%
		-141.958/1.42246,%
		141.958/1.42246,%
		-65.9057/1.34441,%
		65.9057/1.34441%
	}%
	\quad%
	\drawFigure{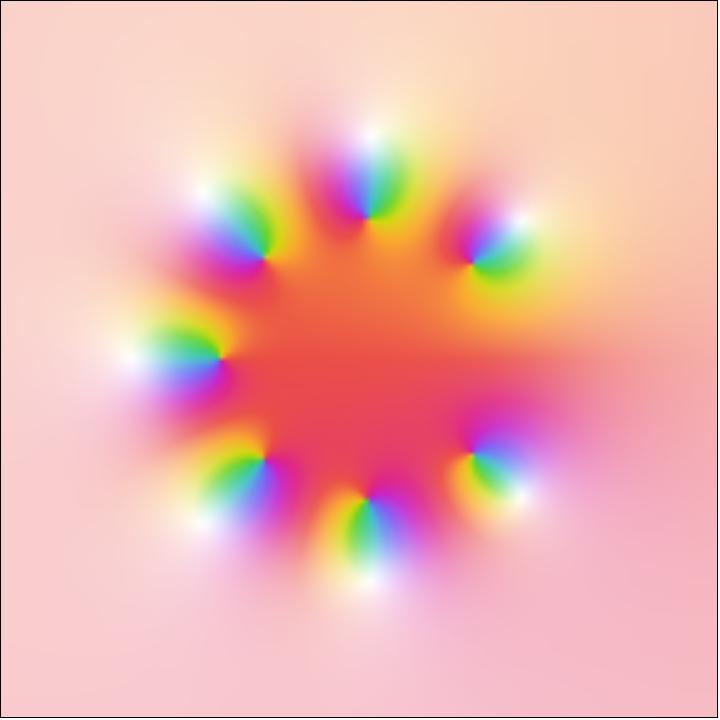}{%
		180/1.28518,%
		-133.456/1.27823,%
		133.456/1.27823,%
		-86.9059/1.25484,%
		86.9059/1.25484,%
		-40.363/1.20159,%
		40.363/1.20159%
	}%
	\quad%
	\drawFigure{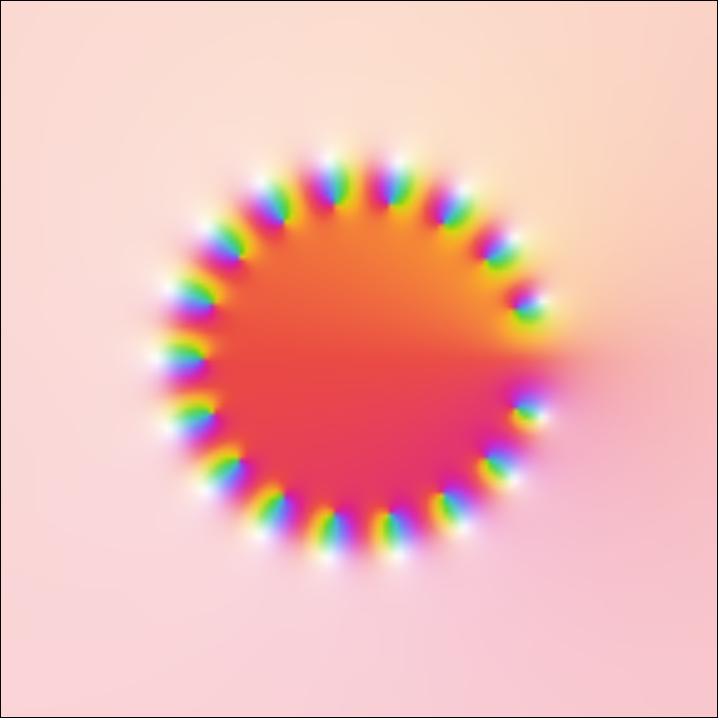}{%
		180/1.14205,%
		-159.707/1.14154,%
		159.707/1.14154,%
		-139.414/1.14,%
		139.414/1.14,%
		-119.121/1.1373,%
		119.121/1.1373,%
		-98.8258/1.13325,%
		98.8258/1.13325,%
		-78.5288/1.12747,%
		78.5288/1.12747,%
		-58.2285/1.1192,%
		58.2285/1.1192,%
		-37.9226/1.10666,%
		37.9226/1.10666,%
		-17.6166/1.084,%
		17.6166/1.084%
	}%
	\caption{\label{fig:LandauExtremizer:PoleBound}%
		Zeros (black dots) and poles (white dots) of $f_n$ from~\zcref{eq:LandausBound:Extremizer} for $n\in\set{4,7,17}$.\\
		In each picture the two circles are centred about~$0$ and have radii $1$ and $1 + 1/(2n+1)$, respectively
	}%
\end{figure*}
\endgroup%
\fi%

\begin{cor}\label{cor:LandauExtremizer:PoleBound}
	The poles $\zeta\in\CC$ of $f_n$ given in~\zcref{eq:LandausBound:Extremizer} satisfy
	\(\displaystyle
		\abs{\zeta} \geq 1 + \frac{1}{2n+1}
	\).
\end{cor}
\begin{proof}
	This follows from the definition~\zcref{eq:LandausBound:Extremizer} of $f_n$,~\zcref{eq:BinomialCoefficient}, \zcref{lem:EnestromKakeya}, and
	\[
		- \binom{-1/2}{\nu} \bigg/ \binom{-1/2}{\nu+1} = \frac{2\nu+2}{2\nu+1}.
		\qedhere
	\]
\end{proof}

From the above one easily obtains decay estimates for the Taylor coefficients of the function $f_n$.
An unpleasant feature of this rather na\"ive approach is that it does not yield an explicit dependence on $n$.
This is a nuisance for our purpose, which we mitigate in \zcref{lem:LandauExtremizer:PowerSeriesExpansion:CoeffGrowth:Explicit}.

\begin{cor}\label{lem:LandauExtremizer:PowerSeriesExpansion:CoeffGrowth}
	Fix an integer $n\geq 1$ and let
	\begin{equation}
		\label{eq:LandauExtremizer:PowerSeriesExpansion}
		f_n(z) = \sum_{\nu=0}^\infty b_{n,\nu} z^\nu
	\end{equation}
	be the power series expansion (about the origin) of the rational function $f_n$ from~\zcref{eq:LandausBound:Extremizer}.
	Then, for every $\epsilon\in(0,1)$,
	\begin{equation}
		\label{eq:LandauExtremizer:PowerSeriesExpansion:CoeffGrowth}
		b_{n,\nu} \ll_{n,\epsilon} \parentheses*{1+\frac{1}{2n+1}-\epsilon}^{-\nu}.
	\end{equation}
\end{cor}
\begin{proof}
	The radius $R$ of convergence of the power series expansion~\zcref{eq:LandauExtremizer:PowerSeriesExpansion} is precisely $\min_{\zeta} \mkern -2mu {\abs{\zeta}}$, where $\zeta$ ranges over the poles of $f_n$.
	The lemma now follows upon combinding~\zcref{cor:LandauExtremizer:PoleBound} and the well-known Cauchy--Hadamard formula for $R$.
\end{proof}

\begin{proof}[Proof of \zcref{thm:Quantitative}]
	Suppose for the moment that we had a version of~\zcref{eq:LandauExtremizer:PowerSeriesExpansion:CoeffGrowth} at our disposal in which the dependence of the implied constant was spelled out in terms of $n$.
	Precisely, we shall work with the inequality
	\begin{equation}\label{eq:ExplicitCoeffBound}
		\abs{b_{n,\nu}} \ll n^2 \parentheses*{1+\frac{1}{4n}}^{-\nu}.
	\end{equation}
	which we establish in \zcref{lem:LandauExtremizer:PowerSeriesExpansion:CoeffGrowth:Explicit} below.
	
	We shall employ~\zcref{eq:ExplicitCoeffBound} to see how well the polynomial
	\[
		f_{n,d} = \sum_{\nu=0}^{d-1} b_{n,\nu} X^\nu.
	\]
	approximates $f_n$ (see~\zcref{eq:LandauExtremizer:PowerSeriesExpansion}).
	Indeed, for all $\abs{z}\leq1$,
	\[
		\abs{ f_n(z) - f_{n,d}(z) }
		= \abs[\bigg]{ \sum_{\nu=d}^\infty b_{n,\nu} z^\nu }
		\ll n^3 \parentheses*{1+\frac{1}{4n}}^{-d}
		\ll n^3 \exp\parentheses*{-\frac{d}{5n}}
		\eqqcolon E(n,d).
	\]
	In particular,
	\[
		\norm{f_{n,d}}_\infty
		\leq \norm{f_n}_\infty + O(E(n,d))
		= 1 + O(E(n,d)).
	\]
	Therefore, $\tilde{f}_{n,d} = f_{n,d} / \norm{f_{n,d}}_\infty$ is a polynomial with $\norm{\tilde{f}_{n,d}}_\infty = 1$.
	If $d>n$, then the first $n+1$ coefficients of the polynomial $f_{n,d}$ sum to the right hand side of~\zcref{eq:LandausBound}.
	In particular, we have constructed a polynomial of degree $<d$, bounded by $1$ in absolute value on the unit disk, whose first $n+1$ coefficients sum to
	\[
		\frac{1}{\norm{f_{n,d}}_\infty} \mathscr{L}_n
		\geq \frac{1}{1 + O(E(n,d))} \mathscr{L}_n
		\geq \parentheses[\big]{ 1 - O\parentheses{E(n,d)} } \mkern 2mu \mathscr{L}_n
	\]
	in absolute value.
\end{proof}

\subsection{A quantitative Eneström--Kakeya theorem}

For our desired explicit version of \zcref{lem:LandauExtremizer:PowerSeriesExpansion:CoeffGrowth}, the original Eneström--Kakeya theorem is insufficient.
Fortunately, it is an easy matter to extract a suitable quantitative statement of the standard proof of said theorem.
Despite its easy proof, we believe that the result obtained may be of independent interest:
\begin{thm}[{Quantitative Eneström--Kakeya}]
	\label{lem:EnestromKakeya:Quantitative}
	Let $p = a_n X^n + \ldots + a_1 X + a_0$ be a polynomial with positive coefficients $a_0 \geq a_1 \geq \ldots \geq a_n > 0$.
	Let $Q = \set{ a_\nu / a_{\nu+1} }[ 0\leq \nu<n ]$ and put $r = \min Q \leq \max Q = R$.
	Then, for $z\in\CC\setminus\set{r,R}$, we have the following inequalities:
	\begin{enumerate}[label={\normalfont(\arabic*)}, ref={(\arabic*)}]
		\item\label{enum:EnestromKakeya:Quantitative:Combined:Inner} \(\displaystyle
			\frac{r-\abs{z}}{\abs{r-z}} \leq \frac{\abs{p(z)}}{a_0} \leq \frac{r+\abs{z}}{\abs{r-z}}
		\) for $\abs{z}\leq r$,
		\smallskip
		\item\label{enum:EnestromKakeya:Quantitative:Combined:Outer} \(\displaystyle
			\frac{\abs{z}-R}{\abs{z-R}} \leq \frac{\abs{p(z)}}{a_n \abs{z}^n} \leq \frac{R+\abs{z}}{\abs{R-z}}
		\) for $\abs{z}\geq R$.
	\end{enumerate}
\end{thm}
\begin{proof}
	(We follow \cite[pp.~12--13]{borwein1995} with minimal adjustments in order to extract the desired inequalities.)
	For $\zeta\in\CC$ with $\abs{\zeta}\leq1$ we have
	\begingroup%
	\allowdisplaybreaks%
	\begin{equation}
		\label{eq:EnestromKakeya}
		\begin{aligned}
			\abs{ (1-\zeta) p(r\zeta) - a_0 } &
			= \abs{ (ra_1-a_0)\zeta+\ldots+(ra_n-a_{n-1})r^{n-1}\zeta^n-a_nr^n\zeta^{n+1} } \\ &
			\leq (a_0-ra_1)\abs{\zeta}+\ldots+(a_{n-1}-ra_n)r^{n-1}\abs{\zeta}^n+a_nr^n\abs{\zeta}^{n+1} \\ &
			\leq \braces{ (a_0-ra_1)+\ldots+(a_{n-1}-ra_n)r^{n-1}+a_nr^n } \abs{\zeta} \\ &
			= a_0 \abs{\zeta}.
		\end{aligned}
	\end{equation}
	\endgroup%
	Then $\zeta = z/r \: (\neq1)$ satisfies $\abs{\zeta}\leq1$ and we find that
	\[
		\abs{ (1-\zeta) p(r\zeta) } \geq a_0 (1-\abs{\zeta})
		\quad\text{and}\quad
		\abs{ (1-\zeta) p(r\zeta) } \leq a_0 (1+\abs{\zeta}).
	\]
	This shows~\zcref{enum:EnestromKakeya:Quantitative:Combined:Inner}.
	
	For~\zcref{enum:EnestromKakeya:Quantitative:Combined:Outer} one considers $\zeta=z/R$ for $\abs{z}\geq R$ and estimates
	\[
		\abs{(1-\zeta)p(R\zeta)+a_n(R\zeta)^n\zeta}
	\]
	similarly as in~\zcref{eq:EnestromKakeya}.
\end{proof}

\subsection{Growth of the Taylor coefficients of Landau's extremal functions\texorpdfstring{ $f_{n,d}$}{}}

Recall that our proof of \zcref{thm:Quantitative}, given in~\zcref{sec:thm:Quantitative:Proof}, crucially hinges on a certain `upgrade' of \zcref{lem:LandauExtremizer:PowerSeriesExpansion:CoeffGrowth} in which the dependence of the bound~\zcref{eq:LandauExtremizer:PowerSeriesExpansion:CoeffGrowth} on $n$ is made explicit (see~\zcref{eq:LandauExtremizer:PowerSeriesExpansion:CoeffGrowth}).
We are now ready to prove this explicit version:

\begin{prop}\label{lem:LandauExtremizer:PowerSeriesExpansion:CoeffGrowth:Explicit}
	Assume the hypotheses of \zcref{lem:LandauExtremizer:PowerSeriesExpansion:CoeffGrowth}.
	Then
	\begin{equation}
		\label{eq:LandauExtremizer:PowerSeriesExpansion:CoeffGrowth:Explicit}
		\abs{b_{n,\nu}} \ll n^2 \parentheses*{1+\frac{1}{4n}}^{-\nu}.
	\end{equation}
\end{prop}
\begin{proof}
	Recall that, by \zcref{cor:LandauExtremizer:PoleBound}, $f_n$ is holomorphic in the open disk with radius $r = 1 + 1/(2n+1)$ centred about the origin.
	Consequently, for any $0 < \rho < r$, Cauchy's estimate for Taylor coefficients yields
	\[
		\abs{b_{n,\nu}}
		\leq \rho^{-\nu} \max_{\abs{z}=\rho} \abs{f_n(z)}
		= \rho^{n-\nu} \max_{\abs{z}=\rho} \abs*{\frac{P_n(1/z)}{P_n(z)}}.
	\]
	We take $\rho = 1 + 1/4n$ and estimate this further using \zcref{lem:EnestromKakeya:Quantitative}.
	Then~\zcref{eq:LandauExtremizer:PowerSeriesExpansion:CoeffGrowth:Explicit} follows easily.
\end{proof}

\subsection{Remark on the zeros of \texorpdfstring{$P_n$}{Pn}}
\label{subsec:RemarksOnZeros}

As we have seen, information on the location of the roots of $P_n$ is vital to our proof of \zcref{thm:Quantitative}.
Precisely, we should like to have good lower bounds on the absolute value of the roots of $P_n$.
Let
\[
	r(n) = \min\set{ \abs{z} }[ \text{roots \(z\) of \(P_n\)} ].
\]
Then \zcref{cor:LandauExtremizer:PoleBound} tells us that $r(n) - 1 \gg 1/n$.
Towards the other extreme, we can give the following upper bound:
\begin{prop}
	We have
	\(\displaystyle
		r(n) \leq 1 + O\parentheses*{\frac{\log n}{n}}
	\)
	for $n\geq 2$.
\end{prop}
\begin{proof}
	(The proof is based on an observation we have learned from work by Erd\H{o}s and Tur\'an~\cite[§\,6]{erdos1950}.)
	Observe that
	\[
		1
		= \abs{P_n(0)}
		= \abs*{\binom{-1/2}{n}} \mkern 2mu \prod_z {\abs{0-z}}
		\geq \abs*{\binom{-1/2}{n}} \mkern 2mu r(n)^n,
	\]
	where $z$ ranges over all $n$ roots of $P_n$ (taking multiplicity into account\footnote{%
		As observed subsequently in \zcref{rem:Roots}, the statement about multiplicities is unnecessary.
	}).
	The absolute value of the binomial coefficient above is asymptotic to $1/\sqrt{\pi n}$ (by Stirling's formula; \emph{cf}.~\cite[pp.~22--23]{landau1916}).
	In particular, we have the inequality $r(n)^n \leq 2 \sqrt{\pi n} < n$ for all sufficiently large $n > 4\pi$.
	Hence, for those $n$,
	\[
		r(n)
		< \exp\parentheses*{\frac{\log n}{n}}
		< 1 + \frac{\log n}{n}.
	\]
	In particular, we have the bound claimed in the proposition for \emph{all} $n\geq 2$.
\end{proof}

Consequently, for $n\geq2$, we have
\begin{equation}\label{eq:RootBounds}
	\frac{1}{n} \ll r(n) - 1 \ll \frac{\log n}{n}.
\end{equation}
Based on some limited numerical evidence, one may be tempted to conjecture that the lower bound is closer to the truth; we have been unable to resolve this.
\begin{opprob}
	Improve on either of the two bounds in~\zcref{eq:RootBounds} to the extent that one of the following two statements follows:
	\[
		\limsup_{n\to\infty} (r(n)-1)n = \infty,
		\quad
		\liminf_{n\to\infty} \frac{(r(n)-1)n}{\log n} = 0.
	\]
	Ideally, establish an asymptotic formula for $(r(n)-1)n$ as $n\to\infty$.
\end{opprob}

\begin{rem}\label{rem:Roots}
	We mention two other interesting facts about the roots of $P_n$:
	\item
	\begin{enumerate}
		\item they are all simple, as can be seen by applying the Eneström--Kakeya theorem to the derivative of $P_n$;
		\item every point on the unit circle is an accumulation point of zeros of $P_n$ for $n$ restricted to a suitable subsequence of the positive integers; this follows from a classical theorem of Jentzsch~\cite{jentzsch1917} on sections of power series with finite radius of convergence (see also the subsequent influential refinements by Szegö~\cite{szego1922} and Erd\H{o}s--Tur\'an~\cite{erdos1950}, as well as the monography~\cite{andrievskii2002}).
	\end{enumerate}
\end{rem}

\section{Interpolation with roots of unity}
\label{sec:Lagrange}

This section is devoted to the proof of both \zcref{thm:Newman:Gen} and \zcref{thm:Newman:Gen:Variant}.
We handle both simultaneously.
Throughout this section fix a positive integer $n$ and let
\begin{equation}
	\label{eq:SetOmega}
	\Omega = \set{ \omega\in\CC }[ \omega^n\pm1 = 0 ].
\end{equation}
In the sequel, $\omega$, $\varpi$ and $\xi$ will always denote elements of $\Omega$.
Occurences of `$\pm$' or `$\mp$' always refer to the sign chosen in~\zcref{eq:SetOmega}.

\subsection{Lagrange interpolators for roots of unity}
Write
\begin{equation}
	\label{eq:LagrangeInterpolators}
	L_{\Omega,\omega} = \prod_{\substack{ \varpi\in\Omega \\ \varpi\neq\omega }} \frac{X-\varpi}{\omega-\varpi}
\end{equation}
for the \emph{Lagrange interpolating polynomial} with respect to $\Omega$ and $\omega$.
Clearly, for $\varpi\in\Omega$,
\begin{equation}
	\label{eq:LagrangeInterpolators:Eval}
	L_{\Omega,\omega}(\varpi) = \begin{cases}
		1 & \text{if } \varpi = \omega, \\
		0 & \text{if } \varpi \neq \omega. \\
	\end{cases}
\end{equation}

The key observation is the following neat property of the Lagrange interpolators with respect to our special choice for $\Omega$:
\begin{prop}\label{prop:NewmanLagrangeInterpol}
	\(\displaystyle
		\sum_{\omega\in\Omega} \abs{L_{\Omega,\omega}(z)}^2 = 1
	\)
	for all $z\in\CC$ with $\abs{z}=1$.
\end{prop}
\begin{rems*}
	\begin{enumerate}
		\item \zcref{prop:NewmanLagrangeInterpol} for `$+$' in the definition of $\Omega$ is precisely Lemma~2 of~\cite{newman1978}.
		\item \zcref{prop:NewmanLagrangeInterpol} is the crucial ingredient for obtaining \zcref{cor:FunctionalNormComputation} below.
		The latter result is, as we shall see, a common generalisation of \zcref{thm:Newman:Gen} and \zcref{thm:Newman:Gen:Variant}.
		\item The statement of \zcref{prop:NewmanLagrangeInterpol} may become invalid if $\Omega$ from~\zcref{eq:SetOmega} is replaced by some other set, even if that set consists only of roots of unity.
		For instance, for $\Omega = \set{1,\I}$, one easily checks that $2 \sum_{\omega\in\Omega} \abs{L_{\Omega,\omega}(z)}^2 = \abs{z-1}^2+\abs{z-\I}^2$.
		It would seem desirable to have a characterisation of the sets $\Omega$ for which \zcref{prop:NewmanLagrangeInterpol} holds.
	\end{enumerate}
\end{rems*}

We base the proof of \zcref{prop:NewmanLagrangeInterpol} on three lemmas.

\begin{lem}
	\label{lem:RootComputation}
	\begin{enumerate*}[label={\normalfont(\arabic*)}, ref={(\arabic*)}]
		\item\label{enum:RootComputation:1}
		\(\displaystyle
			\mp\frac{n}{\omega} = \prod_{\substack{ \varpi\in\Omega \\ \varpi\neq\omega }} (\omega-\varpi)
		\), \,
		\item
		\(\displaystyle
			\mp\frac{(n-1)n}{\omega^2} = 2 \adjustlimits \sum_{\substack{ \xi\in\Omega \\ \xi\neq\omega }} \prod_{\substack{ \vphantom{\xi} \varpi\in\Omega \\ \varpi\neq\omega,\xi }} (\omega-\varpi)
		\).
	\end{enumerate*}
\end{lem}
\begin{proof}
	The first statement is certainly classical (see, \emph{e.g.}, \cite[VI, §~9, Problem~74]{polya1971}); it follows from
	\[
		\mp\frac{n}{\omega}
		= n\omega^{n-1}
		= \frac{\mathrm{d}}{\dd{X}} (X^n\pm1) \evalbar_{X=\omega}
		= \frac{\mathrm{d}}{\dd{X}} \prod_{\varpi\in\Omega} (X-\varpi) \evalbar_{X=\omega}
		= \prod_{\substack{ \varpi\in\Omega \\ \varpi\neq\omega }} (X-\varpi) \evalbar_{X=\omega}.
	\]
	Lisewise, the second statement follows from a similar evaluation of
	the second derivative of $X^n\pm1$ at $\omega\in\Omega$.
\end{proof}

\begin{lem}
	\label{lem:LagrangeInterpolators}
	\(\displaystyle
		L_{\Omega,\omega} = \pm \frac{X^n\pm1}{\omega-X} \frac{\omega}{n}
	\)
	in $\CC(X)$.
\end{lem}
\begin{proof}
	The product over the denominators in~\zcref{eq:LagrangeInterpolators} is $\mp\omega/n$ by \zcref{lem:RootComputation}~\zcref{enum:RootComputation:1}.
	On the other hand, clearly, the product over the numerators equals $(X^n-1)/(X-\omega)$.
	This establishes the lemma.
\end{proof}

\begin{lem}
	\label{lem:PartialFractionDecomp}
	\(\displaystyle
		\sum_{\omega\in\Omega} \frac{\omega}{(X-\omega)^2}
		= \mp \frac{n^2 X^{n-1}}{(X^n\pm1)^2}
	\)
	in $\CC(X)$.
\end{lem}
\begin{proof}
	The claimed identity clearly follows from the polynomial identity
	\begin{equation}
		\label{eq:PartialFractionDecomp}
		\sum_{\omega\in\Omega} \omega \prod_{\substack{ \varpi\in\Omega \\ \varpi\neq\omega }} (X-\varpi)^2 = \mp n^2 X^{n-1}.
	\end{equation}
	upon dividing both sides by $(X^n\pm1)^2 = \prod_{\omega\in\Omega} (X-\omega)^2$.\\
	In order to establish~\zcref{eq:PartialFractionDecomp}, observe that both sides equal one another for all $\omega\in\Omega$, by virtue of~\zcref{lem:RootComputation}~\zcref{enum:RootComputation:1}.
	Furthermore, also the derivatives of both sides of~\zcref{eq:PartialFractionDecomp} are equal for all $\omega\in\Omega$: indeed, we have
	\begin{align*}
		\frac{\mathrm{d}}{\dd{X}} \sum_{w\in\Omega} w \prod_{\substack{ \varpi\in\Omega \\ \varpi\neq w }} (X-\varpi)^2 \evalbar_{X=\omega} 
		= \omega \sum_{\substack{ \xi\in\Omega \\ \xi\neq\omega }} 2 (\omega-\xi) \prod_{\substack{ \varpi\in\Omega \\ \varpi\neq \omega,\xi }} (\omega-\varpi)^2 
		= \frac{(n-1)n^2}{\omega^2},
	\end{align*}
	where the last equality follows from (both parts of) \zcref{lem:RootComputation}.
	Hence, the difference of both sides of~\zcref{eq:PartialFractionDecomp} either vanishes or is a polynomial of degree $\geq 2n-1$.
	Evidently, the latter is not the case, so we have~\zcref{eq:PartialFractionDecomp}.
\end{proof}

\begin{proof}[Proof of \zcref{prop:NewmanLagrangeInterpol}]
	Clearly, we may suppose that $z\notin\Omega$ (by~\zcref{eq:LagrangeInterpolators:Eval} or by continuity, for instance).
	Since $\abs{z}=1$, it holds that $\overline{z} = 1/z$.
	Thus, by \zcref{lem:LagrangeInterpolators},
	\[
		\sum_{\omega\in\Omega} \abs{L_{\Omega,\omega}(z)}^2
		= \sum_{\omega\in\Omega} \frac{z^n\pm1}{\omega-z} \frac{\omega}{n} \frac{z^{-n}\pm1}{\omega^{-1}-z^{-1}} \frac{\omega^{-1}}{n}
		= \mp \frac{(z^n\pm1)^2}{n^2 z^{n-1}} \sum_{\omega\in\Omega} \frac{\omega}{(z-\omega)^2}.
	\]
	The result is now immediate from \zcref{lem:PartialFractionDecomp}.
\end{proof}

\subsection{Application to functionals: the key result}

In the sequel, we shall write $\Ev{\omega}$ for the \emph{evaluation-at-$\omega$ functional} $\CC[X]\to\CC$, mapping every polynomial $f$ to $f(\omega)$.
\zcref{prop:NewmanLagrangeInterpol} now affords a proof of the following result which turns out to be nothing but a unified version of \zcref{thm:Newman:Gen} and \zcref{thm:Newman:Gen:Variant}:

\begin{cor}\label{cor:FunctionalNormComputation}
	Let $n$ and $\Omega$ be as above and let $d \geq 2n-1$ be an arbitrary integer.
	Let $\ell\colon \CC[X]^{<d} \to \CC$ be a functional of the form
	\begin{equation}
		\label{eq:FunctionalPointEvalDecomp}
		\ell = \sum_{\omega\in\Omega} u_\omega \cdot \Ev{\omega}\rvert_{\CC[X]^{<d}}
	\end{equation}
	with scalars $u_\omega\in\CC$.
	Then the norm
	\[
		\norm{\ell} = \sup\set{ \abs{\ell(f)} }[ f\in\CC[X]^{<d},\, \norm{f}_\infty = 1 ]
	\]
	of $\ell$ is given by
	\[
		\norm{\ell} = \sum_{\omega\in\Omega} \abs{ u_\omega }.
	\]
\end{cor}
\begin{proof}
	Let $f\in\CC[X]^{<d}$ satisfy $\norm{f}_\infty = 1$.
	Certainly,
	\[
		\abs{\ell(f)}
		\leq \sum_{\omega\in\Omega} \abs{u_\omega} \abs{f(\omega)}
		\leq \sum_{\omega\in\Omega} \abs{u_\omega},
	\]
	so it suffices to produce an $f$ with $\norm{f}_\infty = 1$ for which the above inequalities hold as equalities.
	Such an $f$ is furnished by
	\[
		f = \sideset{}{^\star}\sum_{\omega\in\Omega} \frac{\overline{u_\omega}}{\abs{u_\omega}} L_{\Omega,\omega}^2,
	\]
	where $\sum^\star$ indicates that terms with $u_\omega = 0$ ought to be omitted.
	Indeed, since each $L_{\Omega,\omega}^2$ has degree $2n-2 < d$, the degree of $f$ is no higher than that.
	Moreover, $\norm{f}_\infty = 1$ by \zcref{prop:NewmanLagrangeInterpol}, and~\zcref{eq:LagrangeInterpolators:Eval} shows that
	\[
		\ell(f)
		= \sideset{}{^\star}\sum_{\omega\in\Omega} \frac{\overline{u_\omega}}{\abs{u_\omega}} u_\omega
		= \sum_{\omega\in\Omega} \abs{u_\omega}.
		\qedhere
	\]
\end{proof}

\begin{rem}
	In~\zcref{subsec:ShapiroExact} we have already noted that the idea of maximising certain functionals by relating them to point evaluation functionals is crucial in~\cite{shapiro1952,newman1978}.
	The interested reader is also referred to~\cite{rogosinski1962} for general reflections on such an approach or, \emph{e.g.}, \cite[Ch.~2.B]{rivlin1990} for further examples.
\end{rem}

\subsection{Application to functionals: cosmetic surgery}
\label{subsec:ApplicationToFunctionals:CosmeticSurgery}

The functionals $\ell$ to which \zcref{cor:FunctionalNormComputation} applies are given in~\zcref{eq:FunctionalPointEvalDecomp} in terms of point evaluation functionals.
For applications to~\zcref{eq:FunctionalAbsValue:General} we should rather like to describe these functionals in terms of their action on monomials.
This is a staight-forward exercise in linear algebra and is the object of the present subsection.

\begin{lem}\label{lem:FunctionalPointEvalDecomp:Criterion}
	A linear functional $\ell\colon\CC[X]^{<d}\to\CC$ is of the form~\zcref{eq:FunctionalPointEvalDecomp} if and only if $(u_\omega)_{\omega\in\Omega}$ satisfies the following system of linear equations:
	\begin{equation}
		\label{eq:LinearSystem}
		\sum_{\omega\in\Omega} u_\omega \omega^\kappa = \ell(X^\kappa),
		\quad0\leq\kappa<d.
	\end{equation}
\end{lem}
\begin{proof}
	Consider the spaces $V = \CC[X]^{<n}$ and $W = \CC[X]^{<d}$.
	On $V$ pick the basis $\set{X^0,X^1,\ldots,X^{n-1}}$ and consider its dual basis $\set{x_0^*,x_1^*,\ldots,x_{n-1}^*} \subset V^*$ (\emph{i.e.}, $x_\nu^*(X^\eta) = 1$ if $\nu=\eta$ and $=0$ otherwise).
	Construct a basis $\set{\tilde{x}_0^*,x_1^*,\ldots,\tilde{x}_{d-1}^*} \subset W^*$ in the same fashion.
	Now put the elements of $\Omega$ in any order, $\Omega = \set{\omega_0,\ldots,\omega_{n-1}}$, say, and consider the linear map $T\colon V^*\to W^*$ sending each $x_\nu^*$ to $\Ev{\omega_\nu}\rvert_W$.
	Observe that any $\ell\in W^*$ admits the decomposition
	\begin{equation}
		\label{eq:FunctionalDualBasisDecomp}
		\ell = \sum_{\kappa=0}^{d-1} \ell(X^\kappa) \mkern 2mu \tilde{x}_\kappa^*
	\end{equation}
	with respect to our chosen dual basis on $W^*$.
	From this we easily obtain a matrix representation of $T$; this is summed up in the following commutative diagram:
	\begin{equation}\label{eq:CommDiagram}
		\begin{tikzcd}[column sep=35mm, row sep=12mm, ampersand replacement=\&] 
			|[alias=Cn  ]| {\CC^n}       \&
			|[alias=CXn ]| {\CC[X]^{<n}} \&
			|[alias=V   ]| V^*           \\
			|[alias=Cd  ]| {\CC^d}       \&
			|[alias=CXd ]| {\CC[X]^{<d}} \&
			|[alias=W   ]| W^*           \mathrlap{.}
			\ar[from=Cn, to=CXn, "\sim"{sloped, inner sep=1pt}, "{ (0,\ldots,0,\underset{\mathclap{\nu}}{1},0,\ldots,0)\mapsto X^\nu }"']
			\ar[from=CXn, to=V, "\sim"{sloped, inner sep=1pt}, "{ X^\nu \mapsto x_\nu^* }"']
			\ar[from=Cd, to=CXd, "\sim"{sloped, inner sep=1pt}, "{ (0,\ldots,0,\underset{\mathclap{\kappa}}{1},0,\ldots,0)\mapsto X^\kappa }"']
			\ar[from=CXd, to=W, "\sim"{sloped, inner sep=1pt}, "{ X^\kappa \mapsto \tilde{x}_\kappa^* }"']
			\ar[from=Cn, to=Cd, "\vphantom{X^\nu}\smash{ \parentheses[\big]{\:\omega_\nu^\kappa\:}_{\substack{ \scriptscriptstyle 0\leq\kappa<d \\ \scriptscriptstyle 0\leq\nu<n \hfill }} }", hook]
			\ar[from=CXn, to=CXd, "{ X^\nu \mapsto \mathop{\smash{\sum\limits_{\kappa=0}^{d-1}}} \omega_\nu^\kappa X^\kappa }", hook]
			\ar[from=V, to=W, "T\strut"', "\strut{ x_\nu^* \mapsto \Ev{\omega_\nu}\rvert_W }"{overlay}, hook]
		\end{tikzcd}
	\end{equation}
	Now $\ell\in W^*$ belongs to the image of $T$ if and only if there is a decomposition of the shape~\zcref{eq:FunctionalPointEvalDecomp}.
	Upon recalling~\zcref{eq:FunctionalDualBasisDecomp} and~\zcref{eq:CommDiagram}, we deduce the asserted equivalency of~\zcref{eq:FunctionalPointEvalDecomp} with~\zcref{eq:LinearSystem}.
\end{proof}

Observe that the left hand side of~\zcref{eq:LinearSystem} is periodic:
\begin{enumerate}
	\item with period $n$ if `$-$' is chosen in the definition~\zcref{eq:SetOmega} of $\Omega$,
	\item with period $2n$ if `$+$' is chosen, and `almost' $n$-periodic in the sense that a negative sign is picked up after $n$ steps (due to $\omega^n=-1$ in that case).
\end{enumerate}
This explains the particular choices for $\boldsymbol{t}$ in \zcref{thm:Newman:Gen} and \zcref{thm:Newman:Gen:Variant}.
In order to solve~\zcref{eq:LinearSystem} explicitly, we shall make use of the following lemma:

\begin{lem}[Orthogonality relations]\label{lem:OrthogonalityRelations}
	Let $\nu$ be an integer. Then
	\[
		\sum_{\omega\in\Omega} \omega^\nu
		= {\begin{cases}
			\epsilon_{\Omega,\nu} n & \text{if } \nu    \equiv0\bmod n, \\
			0 & \text{if } \nu\not\equiv0\bmod n, \\
		\end{cases}}
	\]
	where $\epsilon_{\Omega,\nu}=1$ if the sign `$-$' is chosen in~\zcref{eq:SetOmega} and $\epsilon_{\Omega,\nu}=(-1)^{\nu/n}$ otherwise.
\end{lem}
\begin{proof}
	When the sign `$-$' is chosen in~\zcref{eq:SetOmega}, then the lemma asserts nothing but the well-known \emph{orthogonality relations} for roots of unity (see, \emph{e.g.},~\cite[p.~3]{davenport2000}).
	On the other hand, if the sign `$+$' is chosen in~\zcref{eq:SetOmega}, we may conclude as follows:
	let $\delta_{\nu\equiv0\bmod n} = 1$ or $=0$ according to whether $\nu\equiv0\bmod n$ or not.
	Then, on appealing to the already established `$-$' case,
	\[
		\sum_{\omega\in\Omega} \omega^\nu
		+ n \mkern 2mu \delta_{\nu\equiv0\bmod n}
		= \sum_{\omega^n=-1} \omega^\nu
		+ \sum_{\zeta^n= 1} \zeta^\nu
		= \sum_{\zeta^{2n}=1} \zeta^\nu
		= 2n \mkern 2mu \delta_{\nu\equiv0\bmod2n}.
	\]
	This establishes the lemma.
\end{proof}

\begin{proof}[Proof of \zcref{thm:Newman:Gen} and \zcref{thm:Newman:Gen:Variant}]
	Depending on the sign chosen in~\zcref{eq:SetOmega}, define the linear functional $\ell\colon\CC[X]^{<d}\to\CC$ either via the sum on the left hand sides of~\zcref{eq:NewmanGeneralised:Plus} (\emph{case~`$+$\!'}) or of~\zcref{eq:NewmanGeneralised:Plus} (\emph{case~`$-$\!'}) respectively.
	Then, for $0\leq\nu<d$,
	\begin{equation}
		\label{eq:OurFunctional}
		\ell(X^\nu) = {\begin{cases}
			t_{\nu\bmod 2n} & \text{in case~`\(+\)',} \\
			t_{\nu\bmod  n} & \text{in case~`\(-\)'.} \\
		\end{cases}}
	\end{equation}
	Let $m$ be defined to be either $2n$ or $n$ depending on the sign chosen in~\zcref{eq:SetOmega}.
	We aim to apply \zcref{cor:FunctionalNormComputation} and, to this end, we contend that $\ell$ admits a decomposition of the form~\zcref{eq:FunctionalPointEvalDecomp}.
	By \zcref{lem:FunctionalPointEvalDecomp:Criterion} this means that we ought to solve the linear equations given in~\zcref{eq:LinearSystem}.
	In fact, we claim that the (unique) solution is given by
	\[
		u_\omega = \frac{1}{n} \sum_{\nu=0}^{n-1} \frac{t_{\nu\bmod m}}{\omega^\nu}
		\quad(\omega\in\Omega).
	\]
	Indeed, for $0\leq\kappa<d$, we have
	\[
		\sum_{\omega\in\Omega} u_\omega \omega^{\kappa}
		= \sum_{\omega\in\Omega} \frac{1}{n} \sum_{\nu=0}^{n-1} t_{\nu\bmod m} \omega^{\kappa-\nu}
		= \sum_{\nu=0}^{n-1} t_{\nu\bmod m} \frac{1}{n} \sum_{\omega\in\Omega} \omega^{\kappa-\nu}
		= \ell(X^{\kappa}),
	\]
	where the last equality follows from \zcref{lem:OrthogonalityRelations} and~\zcref{eq:OurFunctional}.
	This shows that we may apply \zcref{cor:FunctionalNormComputation}, and this completes the proof of the two theorems.
\end{proof}

\subsection{Proof of the corollaries}
\label{subsec:ProofOfCorollaries}

\begin{proof}[Proof of \zcref{cor:Application:CuteIneq}]
	We choose $\boldsymbol{t}^n \in \CC^n$ in \zcref{thm:Newman:Gen:Variant} as an appropriate standard unit vector.
	Then~\zcref{eq:NewmanGeneralised:Minus} yields\footnote{%
		If we had applied~\zcref{eq:NewmanGeneralised:Plus} instead, we would have obtained~\zcref{eq:Application:CuteIneq:Almost} with a `$-$' between the two fractions.
		This would equally well yield to the desired conclusion~\zcref{eq:Application:CuteIneq}.
	}
	\begin{equation}\label{eq:Application:CuteIneq:Almost}
		\abs*{ \frac{f^{(k)}(0)}{k!} + \frac{f^{(n+k)}(0)}{(n+k)!} }
		\leq \norm{f}_\infty.
	\end{equation}
	This is \emph{not quite}~\zcref{eq:Application:CuteIneq}.
	However, on replacing $f$ by $f(\eta X)$, where $\eta$ is some suitably chosen unimodular constant, we may arrange for both numbers $f^{(k)}(0)$ and $f^{(n+k)}(0)$ to have equal argument in the complex plane.
	Hence, \zcref{eq:Application:CuteIneq:Almost} implies~\zcref{eq:Application:CuteIneq}.
	The sharpness of~\zcref{eq:Application:CuteIneq} follows, because~\zcref{eq:Application:CuteIneq:Almost} is guaranteed to be sharp by \zcref{thm:Newman:Gen:Variant}.
\end{proof}

\begin{proof}[Proof of \zcref{cor:NewmanGeneralised:Combined}]
	Note that $D$ is always even.
	We let $n = D/2$ so that either $d = 2n$ or $d = 2n-1$ depending on the parity of $d$.
	In the latter case, define $t_d = t_{2n-1} = 0$.
	Let $\inner{\boldsymbol{t},\_}\colon\CC[X]^{<d}\to\CC$ denote the linear functional sending each $f = \sum_\nu a_\nu X^\nu \in\CC[X]^{<d}$ to $\inner{\boldsymbol{t},f} = t_0 a_0 + \ldots + t_{d-1} a_{d-1}$.
	We decompose $\boldsymbol{t}$ as follows:
	\begin{equation}
		\label{eq:NewmanGeneralised:Combined:Proof:tVectDecomp}
		\boldsymbol{t}
		= {\begin{pmatrix}
			t_0 \\ \vdots \\ t_{n-1} \\ t_n \\ \vdots \\ t_{d-1}
		\end{pmatrix}}
		= {\frac{1}{2} \begin{pmatrix}
			t_0+t_n \\ \vdots \\ t_{n-1}+t_{2n-1} \\ t_0+t_n \\ \vdots \\ t_{d-n-1}+t_{d-1}
		\end{pmatrix}}
		+ {\frac{1}{2} \begin{pmatrix}
			t_0-t_n \\ \vdots \\ t_{n-1}-t_{2n-1} \\ -(t_0-t_n) \\ \vdots \\ -(t_{d-n-1}-t_{d-1})
		\end{pmatrix}}
		\eqqcolon \boldsymbol{t}_{+} + \boldsymbol{t}_{-}.
	\end{equation}
	Hence,
	\begin{equation}
		\label{eq:NewmanGeneralised:Combined:Proof}
		\abs{ t_0 a_0 + \ldots + t_{d-1} a_{d-1} }
		= \abs{\inner{(\boldsymbol{t}_{+}+\boldsymbol{t}_{-}),f}}
		\leq \abs{\inner{\boldsymbol{t}_{+},f}} + \abs{\inner{\boldsymbol{t}_{-},f}}.
	\end{equation}
	The first term on the right hand side of \zcref{eq:NewmanGeneralised:Combined:Proof} can be estimated using \zcref{thm:Newman:Gen:Variant}; likewise, the second term can be estimated using \zcref{thm:Newman:Gen}.
	Specifically, we find that
	\[
		\abs{\inner{\boldsymbol{t}_{\pm},f}}
		\leq \frac{1}{n} \sum_{\omega^n=\pm1} \abs[\bigg]{ \sum_{\nu=0}^{n-1} \frac{1}{2} \frac{t_\nu \pm t_{\nu+n}}{\omega^\nu} }
		= \frac{1}{D} \sum_{\omega^n=\pm1} \abs[\bigg]{ \sum_{\nu=0}^{D-1} \frac{t_\nu}{\omega^\nu} }
		= \frac{1}{D} \sum_{\omega^n=\pm1} \abs[\bigg]{ \sum_{\nu=0}^{d-1} \frac{t_\nu}{\omega^\nu} },
	\]
	where the penultimate equation is clear since either $d=D$ (if $d$ is even) or $d=D-1$ and $t_d = 0$ (by definition).
	Upon recalling~\zcref{eq:NewmanGeneralised:Combined:Proof} and noting that
	\[
		\sum_{\omega^n=1} + \sum_{\omega^n=-1}
		= \sum_{\omega^D=1},
	\]
	we finally infer the estimate~\zcref{eq:NewmanGeneralised:Combined}.
\end{proof}

\section*{Acknowledgements}

This research was funded in whole by the \emph{Austrian Science Fund~(FWF)}, project \href{http://doi.org/10.55776/PAT4579123}{`\emph{Prime divisors of polynomials, spin chains, and non-residues}'} (grant doi: \href{http://doi.org/10.55776/PAT4579123}{\footnotesize 10.55776/PAT4579123}).
For open access purposes, the author has applied a \href{https://creativecommons.org/licenses/by/4.0/}{CC~BY public copyright license} to any author accepted manuscript version arising from this submission.
The author would like to thank the anonymous referee for useful criticism that has lead to an improved exposition.


\begin{thebibliography}{20}
	\bibitem{andrievskii2002}
	V.~V. {Andrievskii} and {\relax H.-P}.~{Blatt}.
	\newblock \href{https://doi.org/10.1007/978-1-4757-4999-1}{\emph{{Discrepancy of signed measures and polynomial approximation}}}.
	\newblock New York, NY: Springer, 2002.
	
	\bibitem{borwein1995}
	P.~{Borwein} and T.~{Erd{\'e}lyi}.
	\newblock \href{https://doi.org/10.1007/978-1-4612-0793-1}{\emph{{Polynomials and polynomial inequalities}}}.
	\newblock New York, NY: Springer, 1995.
	
	\bibitem{davenport2000}
	H.~{Davenport}.
	\newblock \href{https://doi.org/10.1007/978-1-4757-5927-3}{\emph{{Multiplicative number theory. Revised by Hugh L.~Montgomery}}}.
	\newblock New York, NY: Springer, 3rd~ed., 2000.
	
	\bibitem{enestrom1893}
	G.~{Enestr{\"o}m}.
	\newblock {Härledning af en allmän formel för antalet pensionärer, som vid en godtycklig tidpunkt förefinnas inom en sluten pensionskassa}.
	\newblock {Stockh.\ {\"O}fv.~{L}. 405--415}, 1893.
	
	\bibitem{erdos1950}
	P.~{Erd{\H{o}}s} and P.~{Tur{\'a}n}.
	\newblock \href{https://doi.org/10.2307/1969500}{{On the distribution of roots of polynomials}}.
	\newblock \emph{{Ann.\ Math.~(2)}}, 51:\penalty0 105--119, 1950.
	
	\bibitem{jentzsch1917}
	R.~{Jentzsch}.
	\newblock \href{https://doi.org/10.1007/BF02422945}{{Untersuchungen zur Theorie der Folgen analytischer Funktionen}}.
	\newblock \emph{{Acta Math.}}, 41:\penalty0 219--251, 1917.
	
	\bibitem{kakeya1912}
	S.~{Kakeya}.
	\newblock \href{https://doi.org/10.1016/0021-9045(79)90056-X}{{On the limits of the roots of an algebraic equation with positive coefficients}}.
	\newblock \emph{{T{\^o}hoku Math.~J.}}, 2:\penalty0 140--142, 1912.
	
	\bibitem{landau1913}
	E.~{Landau}.
	\newblock {Absch{\"a}tzung der Koeffizientensumme einer Potenzreihe. I, II}.
	\newblock \emph{{Arch.\ der Math.~u.\ Phys.~(3)}}, 21:\penalty0 42--50, 250--255, 1913.
	
	\bibitem{landau1916}
	E.~{Landau}.
	\newblock \href{https://doi.org/10.1007/978-3-642-91869-8}{\emph{{Darstellung und Begr{\"u}ndung einiger neuerer Ergebnisse der Funktionentheorie}}}.
	\newblock Berlin: Springer, 1916.
	
	\bibitem{markov1916}
	V.~A. {Markov}.
	\newblock \href{https://doi.org/10.1007/BF01456902}{\"Uber Polynome, die in einem gegebenen Intervalle m\"oglichst wenig von Null abweichen}.
	\newblock (Transl.\ by J.~Grossmann; with a preface by Serge Bernstein.)
	\newblock \emph{{Math.\ Ann.}}, 77:\penalty0 213--258, 1916.
	
	\bibitem{newman1978}
	D.~J. {Newman}.
	\newblock \href{https://doi.org/10.1016/0021-9045(78)90106-5}{{Polynomials with large partial sums}}.
	\newblock \emph{{J.~Approx.\ Theory}}, 23:\penalty0 187--190, 1978.
	\newblock \href{https://doi.org/10.1016/0021-9045(79)90056-X}{{\emph{Erratum}}} \emph{ibid}.~26:\penalty0 194, 1979.
	
	
	\bibitem{polya1971}
	G.~{P{\'o}lya} and G.~{Szeg\H{o}}.
	\newblock \href{https://doi.org/10.1007/978-3-642-61987-8}{\emph{{Aufgaben und Lehrs{\"a}tze aus der Analysis. 2.~Band: Funktionentheorie, Nullstellen, Polynome, Determinanten, Zahlentheorie}}}.
	\newblock Berlin: Springer, 4th~ed., 1971.
	
	\bibitem{rack1989}
	{\relax H.-J}.~{Rack}.
	\newblock \href{https://doi.org/10.1016/0021-9045(89)90124-X}{On polynomials with largest coefficient sums}.
	\newblock \emph{{J.~Approx.\ Theory}}, 56\penalty0 (3):\penalty0 348--359, 1989.
	
	\bibitem{reimer1967}
	M.~{Reimer} and K.~{Zeller}.
	\newblock \href{https://doi.org/10.1007/BF01123741}{Absch{\"a}tzung der Teilsummen reeller Polynome}.
	\newblock \emph{{Math.\ Z.}}, 99:\penalty0 101--104, 1967.
	
	\bibitem{rivlin1990}
	{\relax Th}.~J. {Rivlin}.
	\newblock \emph{{Chebyshev polynomials. From approximation theory to algebra and number theory}}.
	\newblock New York: John Wiley~\&\ Sons, Inc., 2nd~ed., 1990.
	
	\bibitem{rogosinski1962}
	W.~W. {Rogosinski}.
	\newblock \href{https://doi.org/10.7146/math.scand.a-10513}{{Linear functionals on bounded polynomials of given order}}.
	\newblock \emph{{Math.\ Scand.}}, 10:\penalty0 53--70, 1962.
	
	\bibitem{rudin1987}
	W.~{Rudin}.
	\newblock \emph{{Real and complex analysis}}.
	\newblock New York, NY: McGraw-Hill, 3rd~ed., 1987.
	
	\bibitem{shapiro1951}
	H.~S. {Shapiro}.
	\newblock \href{http://hdl.handle.net/1721.1/12198}{Extremal problems for polynomials and power series}.
	\newblock Master's thesis, Massachusetts Institute of Technology. Department of Mathematics, 1951.
	
	\bibitem{shapiro1952}
	H.~S. {Shapiro}.
	\newblock \emph{\href{https://dspace.mit.edu/handle/1721.1/12247}{Extremal problems for polynomials and power series}}.
	\newblock PhD thesis, Massachusetts Institute of Technology. Department of Mathematics, 1952.
	
	\bibitem{szego1922}
	G.~{Szeg{\"o}}.
	\newblock {\"Uber die Nullstellen von Polynomen, die in einem Kreise gleichm\"a{\ss}ig konvergieren}.
	\newblock \emph{{Sitzungsber.\ Berl.\ Math.\ Ges.}}, 21:\penalty0 59--64, 1922.
\end{thebibliography}


\vfill%
\end{document}